\documentclass[letterpaper, 10 pt, conference]{ieeeconf}
\IEEEoverridecommandlockouts
\usepackage{cite}
\usepackage{graphicx}
\usepackage{epstopdf}
\usepackage{amsmath,amssymb,amsfonts}

\usepackage{amsthm}

\usepackage{algorithmic}
\usepackage{textcomp}
\usepackage{xcolor}
\usepackage{algorithm}

\newtheorem{assumption}{Assumption}
\newtheorem{remark}{Remark}
\newtheorem{lemma}{Lemma}
\newtheorem{theorem}{Theorem}
\newtheorem{definition}{Definition}
 
\newtheorem{corollary}{Corollary}     

\usepackage{balance}
\usepackage{hyperref}
\title{\LARGE \bf
Stochastic Momentum Tracking Push-Pull for Decentralized Optimization over Directed Graphs
}

\author{Wenqi Fan, Yiwei Liao, Qing Xu, Bin Guo, and Songyi Dian
\thanks{This work was supported in part by the National Natural Science Foundation of China under Grant 62503344 and in part by the Postdoctoral Fellowship Program of CPSF under Grant GZC20241120. (Corresponding author: Yiwei Liao.)}%
\thanks{W. Fan, Y. Liao, Q. Xu, and S. Dian are with the School of Electrical Engineering, Sichuan University, Sichuan, 610065, P.R. China.
        {\tt\small (emails: fanjessie159@gmail.com, liaoyiwei@scu.edu.cn, xuqing382015@163.com, bguodxl@163.com, scudiansy@scu.edu.cn)}}%
}
    
\begin{document}

\thispagestyle{empty} 
\pagestyle{empty}     

\maketitle

\begin{abstract}
Decentralized optimization over directed networks is frequently challenged by asymmetric communication and the inherent high variance of stochastic gradients, which collectively cause severe oscillations and hinder algorithmic convergence. To address these challenges, we propose the Stochastic Momentum Tracking Push-Pull (SMTPP) algorithm, which tracks the momentum term rather than raw stochastic gradients within the Push-Pull architecture. This design successfully decouples the variance reduction capacity from the algebraic connectivity of the graph.Although the inherent topology mismatch of directed graphs precludes exact convergence under persistent stochastic noise, SMTPP rigorously compresses this unavoidable steady-state error floor into a minimal neighborhood determined by network connectivity and gradient variance. Furthermore, SMTPP guarantees convergence on any strongly connected directed graph. Extensive experiments on non-convex logistic regression demonstrate that the algorithm is highly robust to network connectivity. By effectively dampening topology-induced oscillations, SMTPP achieves convergence rates and overall performance that closely match those of centralized baselines, regardless of whether the network is sparse or dense.
\end{abstract}


\section{Introduction}

Decentralized optimization has been extensively utilized in sensor networks, deep learning, and smart grids \cite{liu2022survey}, since decentralization helps relieve the high latency at the central server and reduce the communication overhead. Although decentralized stochastic gradient descent (DSGD) \cite{lian2017can} is widely used, its convergence performance is fundamentally influenced by  stochastic noise and data heterogeneity. Gradient tracking  is one popular remedy to mitigate this issue\cite{liao2022compressed, qu2017harnessing, pu2021distributed}.

To further accelerate convergence and dampen stochastic noise, the utilization of momentum \cite{cutkosky2019momentum} has attracted attention in decentralized optimization. Notably, \textit{momentum tracking} methods, which track the global average momentum, have achieved linear speedup in undirected networks \cite{huang2025distributed, gao2020periodic, alghunaim2022unified} and successfully accommodated communication compression \cite{islamov2025motef}. However, most momentum-based decentralized algorithms strictly assume undirected or balanced topologies \cite{huang2024faster, doostmohammadian2020momentum}.

In real-world networks, asymmetric transmission powers often dictate directed communication links. To handle the situation, the \textit{Push-Pull} framework \cite{pu2020push, you2025stochastic} utilizes two distinct directed graphs for pushing gradients and pulling models, becoming a standard for directed decentralized optimization \cite{liang2023linear, xin2018linear}. The extension of Push-Pull based methods has emerged in diverse scenarios, such as asynchronous setting \cite{zhu2024r} , time-varying graphs \cite{nguyen2023accelerated}, and  communication compression \cite{liao2023linearly}.

In non-convex settings with high variance, the inherent sparsity and asymmetry of directed topologies infinitely amplify the raw gradient noise, leading to severe oscillations and compromised performance. How to effectively unleash the variance-reduction power of \textit{momentum tracking} over general directed networks remains a vital challenge problem.

In this paper, we design the Stochastic Momentum Tracking Push-Pull  algorithm (SMTPP) to solve decentralized optimization problem in general directed networks. 
Crucially, it retains the property of tracking the global average momentum. 

When combined with persistent stochastic noise, standard tracking algorithms suffer from amplified oscillations and a large steady-state error floor of $\mathcal{O}(\sigma^2)$. Although fundamental asymmetry mathematically prevents exact convergence to zero with a constant step size, our theoretical analysis reveals that SMTPP gracefully compresses this topology-induced error floor to $\mathcal{O}(\lambda^2\sigma^2)$. By tracking the momentum rather than gradients, SMTPP effectively decouples the noise reduction capacity from the algebraic connectivity of the graph. Extensive experiments on non-convex logistic regression demonstrate that SMTPP outperforms state-of-the-art baselines, drastically suppressing oscillations in  sparse directed networks.

The remainder of the paper is organized as follows. Section II introduces the problem formulation and establishes the preliminary lemmas. Section III presents the main algorithm and the convergence analysis. Section IV provides numerical experiments to validate our theoretical findings. Finally, Section V concludes the paper.

\section{Problem Setup}
This section outlines the optimization problem on directed networks and introduces the necessary assumptions regarding the communication topology and objective functions.

\subsection{Problem Formulation}

Consider a network of $n$ agents interacting on a directed graph $\mathcal{G} = (\mathcal{V}, \mathcal{E})$, where $\mathcal{V} = \{1, \dots, n\}$ is the set of agents and $\mathcal{E} \subseteq \mathcal{V} \times \mathcal{V}$ is the set of edges. An edge $(j, i) \in \mathcal{E}$ indicates that agent $i$ can receive information from agent $j$. 

The agents cooperatively solve the following non-convex stochastic optimization problem:
\begin{equation}
    \min_{x \in \mathbb{R}^d} F(x) := \frac{1}{n} \sum_{i=1}^n f_i(x), \label{eq:problem}
\end{equation}
where $f_i : \mathbb{R}^d \to \mathbb{R}$ is the local smooth loss function accessible only to agent $i$ and $f_i = \mathbb{E}_{\xi_{i}\sim \mathcal{D}_{i}}F_{i}(x;\xi_{i})$. In the stochastic setting, agent $i$ queries a stochastic oracle to obtain an unbiased estimate $g_i(x; \xi_i)$ of the true gradient $\nabla f_i(x)$.

\subsection{Assumptions}

In this subsection, we present the following standard assumptions regarding communication topology, objective functions, and stochastic oracles.

To address the potential asymmetry of communication, we employ two distinct mixing matrices: a row-stochastic matrix $\mathbf{R}$ for consensus and a column-stochastic matrix $\mathbf{C}$ for momentum tracking.
\begin{assumption} \label{ass:graph}
The communication network utilizes two directed graphs $\mathcal{G}_R = (\mathcal{V}, \mathcal{E}_R)$ and $\mathcal{G}_C = (\mathcal{V}, \mathcal{E}_C)$ for pulling parameters and pushing trackers, respectively. The graphs $\mathcal{G}_R$ and $\mathcal{G}_{C^\top}$  each contain at least one spanning tree, and there exists at least one pair of spanning trees that share a common root. The associated mixing matrices $\mathbf{R}, \mathbf{C} \in \mathbb{R}^{n \times n}$ satisfy
\begin{enumerate}
    \item $\mathbf{R}$ is row-stochastic and $\mathbf{C}$ is column-stochastic.
    \item The weights match the graph topology, i.e., $R_{ij} > 0$ if and only if $(j, i) \in \mathcal{E}_R$ or $i = j$, and $C_{ij} > 0$ if and only if $(j, i) \in \mathcal{E}_C$ or $i = j$.
    \item The diagonal elements are strictly positive, i.e., $R_{ii} > 0$ and $C_{ii} > 0$ for all $i \in \mathcal{V}$.
\end{enumerate}
\end{assumption}

\begin{remark}
Assumption \ref{ass:graph} ensures that $\mathbf{R}$ and $\mathbf{C}$ are primitive. According to the Perron-Frobenius theorem, there exist unique positive stationary vectors $\pi_R, \pi_C$ such that $\pi_R^\top \mathbf{R} = \pi_R^\top$ and $\mathbf{C}\pi_C = \pi_C$, with $\pi_R^\top \mathbf{1} = \mathbf{1}^\top \pi_C = 1$. The property above guaranties the geometric contraction of the matrix powers.
\end{remark}

\begin{assumption} \label{ass:func}
For each agent $i$, the local objective function $f_i$ is differentiable and bounded from below. Furthermore, gradients are $L$-Lipschitz continuous, i.e., for all $x, y \in \mathbb{R}^d$
\begin{equation}\label{Lsmooth}
    \| \nabla f_i(x) - \nabla f_i(y) \| \le L \| x - y \|.
\end{equation}
\end{assumption}

\begin{assumption} \label{ass:noise}
The stochastic gradient oracle $g_i(x; \xi_i)$ satisfies the following properties for all $x \in \mathbb{R}^d$
\begin{gather}\label{Assumption:gradient}
    \mathbb{E}_{\xi_i}[g_i(x; \xi_i)] = \nabla f_i(x), \\
    \mathbb{E}_{\xi_i}[\| g_i(x; \xi_i) - \nabla f_i(x) \|^2] \le \sigma^2.
\end{gather}
\end{assumption}

\begin{remark}
Assumption 3 only constrains the intra-node variance $\sigma^2$. Crucially, we do not require bounded data heterogeneity, which distinguishes our analysis from standard DSGD, allowing the proposed algorithm to handle arbitrary Non-IID data distributions.
\end{remark}

\section{Algorithm and Convergence Analysis}

In this section, we present the proposed SMTPP algorithm, analyze its tracking mechanism over directed graphs, and establish its theoretical convergence properties.

\subsection{The SMTPP Algorithm}
The proposed SMTPP algorithm (Algorithm 1) leverages the Push-Pull architecture to enable non-symmetric updates, elegantly bypassing the costly construction of doubly stochastic matrices in directed networks. While standard Push-Pull tracks raw stochastic gradients, which often leading to severe oscillations, SMTPP mitigates this by tracking the \textit{local momentum}. Specifically, each agent $i$ maintains three key variables: the model parameter $x_{i,k}$, the local momentum buffer $m_{i,k}$, and the momentum tracker $v_{i,k}$. 

\begin{algorithm}[htbp]
\caption{Stochastic Momentum Tracking Push-Pull (SMTPP)}
\label{alg:smtpp}
\textbf{Require:} Initial model $x_{i,0} = x_0$, step size $\eta > 0$, momentum coefficient $\lambda \in (0, 1]$, matrices $\mathbf{R}, \mathbf{C}$.
\begin{algorithmic}[1]
\setlength{\abovedisplayskip}{2pt}
\setlength{\belowdisplayskip}{2pt}
\setlength{\abovedisplayshortskip}{0pt}
\setlength{\belowdisplayshortskip}{0pt}
\STATE \textbf{Initialization:} Each agent $i$ sets local momentum and tracker to zero: $m_{i,0} = \mathbf{0}, \quad v_{i,0} = \mathbf{0}$.
\STATE Each agent computes initial gradient $g_{i,0} = \nabla F_i(x_{i,0}; \xi_{i,0})$.
\FOR{iteration $k = 0, 1, \dots, K-1$}
    \FOR{each agent $i \in \{1, \dots, n\}$ in parallel}
        \STATE \textbf{Model Update (Pull):} 
        \begin{equation}
            x_{i,k+1} = \sum_{j=1}^n R_{ij} (x_{j,k} - \eta v_{j,k}) \label{eq:pull}
        \end{equation}
        \begin{equation}
            g_{i,k+1} = \nabla F_i(x_{i,k+1}; \xi_{i,k+1}) \label{eq:grad}
        \end{equation}
        \begin{equation}
            m_{i,k+1} = (1 - \lambda)m_{i,k} + \lambda g_{i,k+1} \label{eq:momentum}
        \end{equation}
        \begin{equation}
            v_{i,k+1} = \sum_{j=1}^n C_{ij}v_{j,k} + (m_{i,k+1} - m_{i,k}) \label{eq:push}
        \end{equation}
    \ENDFOR
\ENDFOR
\end{algorithmic}
\end{algorithm}

\textbf{1) Consensus and Descent (Pull) - Eq. \eqref{eq:pull}.} Agent $i$ pulls the parameters $x_{j,k}$ and trackers $v_{j,k}$ from its in-neighbors. The row-stochastic matrix $\mathbf{R}$ ($\sum_j R_{ij}=1$) ensures that the convex combination preserves the scale of the model, effectively driving the network towards consensus. 

\textbf{2) Variance Reduction via Momentum - Eq. \eqref{eq:momentum}.} Instead of directly transmitting the highly varying stochastic gradient $g_{i,k+1}$, agent $i$ updates an exponential moving average $m_{i,k+1}$. The momentum coefficient $\lambda$ acts as a low-pass filter, significantly dampening the intra-node stochastic noise $\sigma^2$ and providing a smoother descent direction.

\textbf{3) Global Momentum Tracking (Push) - Eq. \eqref{eq:push}.} Agent $i$ pushes its tracker variable $v_{i,k}$ to its out-neighbors using the column-stochastic matrix $\mathbf{C}$ and incorporates the \textit{change} in its local momentum. Crucially, the column-stochasticity of $\mathbf{C}$ ensures that the sum of the tracking variables is conserved during communication. By adding the local momentum difference $m_{i,k+1} - m_{i,k}$, the network average of the trackers precisely matches the global average momentum at every step. This elegantly enables agents to achieve global momentum tracking without requiring a doubly stochastic matrix.

\textbf{Matrix Form.}
Let $\mathbf{X}_k =[x_{1,k}, \dots, x_{n,k}]^\top \in \mathbb{R}^{n \times d}$ collect the model parameters of all agents in iteration $k$. We similarly define the momentum matrix $\mathbf{M}_k \in \mathbb{R}^{n \times d}$, the tracker matrix $\mathbf{V}_k \in \mathbb{R}^{n \times d}$, and the stochastic gradient matrix $\mathbf{G}_k =[g_{1,k}, \dots, g_{n,k}]^\top \in \mathbb{R}^{n \times d}$. 
\begin{align}
    \mathbf{X}_{k+1} &= \mathbf{R}(\mathbf{X}_k - \eta \mathbf{V}_k), \label{eq:matrix_x} \\
    \mathbf{M}_{k+1} &= (1 - \lambda)\mathbf{M}_k + \lambda \mathbf{G}_{k+1}, \label{eq:matrix_m} \\
    \mathbf{V}_{k+1} &= \mathbf{C}\mathbf{V}_k + \mathbf{M}_{k+1} - \mathbf{M}_k. \label{eq:matrix_v}
\end{align}

\subsection{Preliminary Lemmas}

Due to the asymmetric Push-Pull architecture, we rigorously define the consensus average as $\bar{x}_k := \pi_R^\top \mathbf{X}_k$, while the uniform averages for the trackers and momentum are $\bar{v}_k := \frac{1}{n}\mathbf{1}^\top \mathbf{V}_k$ and $\bar{m}_k := \frac{1}{n}\mathbf{1}^\top \mathbf{M}_k$, respectively. Since the non-symmetric matrices $\mathbf{R}$ and $\mathbf{C}$ lack strict one-step contraction under standard norms, we introduce specific induced norms.

\begin{definition}[Induced Weighted Norms]
Given the primitive matrices $\mathbf{R}$ and $\mathbf{C}$, there exist invertible transformation matrices $\mathbf{S}_R, \mathbf{S}_C \in \mathbb{R}^{n \times n}$ such that we can define the weighted vector norms $\|\mathbf{x}\|_R := \|\mathbf{S}_R \mathbf{x}\|_2$ and $\|\mathbf{x}\|_C := \|\mathbf{S}_C \mathbf{x}\|_2$ for any $\mathbf{x} \in \mathbb{R}^n$. For a matrix $\mathbf{X} = [\mathbf{x}^{(1)}, \dots, \mathbf{x}^{(d)}] \in \mathbb{R}^{n \times d}$, the corresponding induced matrix norms are defined as
\begin{equation}
    \|\mathbf{X}\|_R := \sqrt{\sum_{j=1}^d \|\mathbf{x}^{(j)}\|_R^2}, \quad \|\mathbf{X}\|_C := \sqrt{\sum_{j=1}^d \|\mathbf{x}^{(j)}\|_C^2}.
\end{equation}
These norms are equivalent to the Frobenius norm $\|\cdot\|_F$. There exist constants $\delta_R, \delta_C \ge 1$ such that $\|\cdot\|_F \le \|\cdot\|_\nu \le \delta_\nu \|\cdot\|_F$ for $\nu \in \{R, C\}$.
\end{definition}

\begin{lemma}[Strict Geometric Mixing \cite{song2022compressed}]
\label{lem:mixing}
Under Assumption 1, let $\Pi_R = \mathbf{1}\pi_R^\top$ and $\Pi_C = \pi_C \mathbf{1}^\top$. There exist contraction factors $\rho_R, \rho_C \in (0, 1)$ such that for any matrix $\mathbf{X} \in \mathbb{R}^{n \times d}$ satisfies
\begin{align}
    \|\mathbf{R}\mathbf{X} - \Pi_R \mathbf{X}\|_R &\le \rho_R \|\mathbf{X} - \Pi_R \mathbf{X}\|_R, \\
    \|\mathbf{C}\mathbf{X} - \Pi_C \mathbf{X}\|_C &\le \rho_C \|\mathbf{X} - \Pi_C \mathbf{X}\|_C.
\end{align}
To fully decouple the system dynamics, we define three fundamental error metrics: the consensus error $\mathcal{E}_{x,k} := \mathbb{E}\|\mathbf{X}_k - \Pi_R\mathbf{X}_k\|_R^2$, the tracking error $\mathcal{E}_{v,k} := \mathbb{E}\|\mathbf{V}_k - \Pi_C\mathbf{V}_k\|_C^2$, and the momentum approximation error $\mathcal{E}_{m,k} := \mathbb{E}\|\mathbf{M}_k - \nabla\mathbf{F}(\mathbf{X}_k)\|_F^2$, where $\nabla\mathbf{F}(\mathbf{X}_k) \in \mathbb{R}^{n \times d}$ is the full gradient matrix.
\end{lemma}

\begin{proof}
See Appendix \ref{app:proof_lemma1}
\end{proof}

\begin{lemma}[Conservation Property]
\label{lem:conservation}
Under Algorithm 1, initialized with $v_{i,0} = m_{i,0}$, the column-stochasticity of $\mathbf{C}$ ensures that the average tracker effectively tracks the average momentum at every step
\begin{equation}
    \bar{v}_k = \bar{m}_k, \quad \forall k \ge 0.
\end{equation}
\end{lemma}
\begin{proof}
    Left-multiplying the tracker update $\mathbf{V}_{k+1} = \mathbf{C}\mathbf{V}_k + \mathbf{M}_{k+1} - \mathbf{M}_k$ by $\frac{1}{n}\mathbf{1}^\top$, since $\mathbf{C}$ is column stochastic, we obtain the network average dynamics $\bar{v}_{k+1} - \bar{v}_k = \bar{m}_{k+1} - \bar{m}_k$. The sum of $t=0$ to $k$ together with zero initialization ($\bar{v}_0 = \bar{m}_0 = 0$) inherently guaranties that $\bar{v}_{k} = \bar{m}_{k}$ is valid for all $k \ge 0$. 
\end{proof}

Based on the smoothness of the objective function and the conservation property, we derive the descent inequality.

\begin{lemma} [Descent Inequality]
\label{lem:descent inequality}
Under Assumptions 1 and 2, define the virtual sequence $\tilde{z}_k := \bar{x}_k - \frac{\eta c_\pi (1-\lambda)}{\lambda} \bar{m}_{k-1}$. For $\eta \le \frac{1}{2 L c_\pi}$, the expected objective value evolves as
\begin{align}
\mathbb{E}[F(\tilde{z}_{k+1})] \le &\, \mathbb{E}[F(\tilde{z}_k)] - \frac{\eta c_\pi}{8} \mathbb{E}\|\nabla F(\bar{x}_k)\|^2 \notag \\
& + \frac{L^2 \eta c_\pi}{n} \mathcal{E}_{x,k} + c_z \eta \mathcal{E}_{v,k} \notag \\
& + \frac{5L^2 c_\pi^3 \eta^3}{4\lambda^2} \mathbb{E}\|\bar{m}_{k-1}\|^2 + \frac{L c_\pi^2 \eta^2 \sigma^2}{n}.
\end{align}
where $c_\pi = n \pi_R^\top \pi_C > 0$ is the tracking projection constant, and $c_z$ accounts for the projection mismatch between the row-stochastic and column-stochastic graphs.
\end{lemma}

\begin{proof}
See Appendix \ref{app:proof_lemma3}
\end{proof}
\begin{remark}
The virtual sequence $\tilde{z}_k$ is carefully constructed to mimic a scaled SGD step, which isolates the true stochastic gradient while explicitly capturing the topology-induced tracking mismatch for the subsequent Lyapunov analysis.
\end{remark}

\begin{lemma}[Consensus Error Bound]
\label{lem:consensus}
Under Assumption 1, for any $\eta > 0$, the consensus error satisfies
\begin{equation}
    \mathcal{E}_{x,k+1} \le \frac{1 + \rho_R^2}{2} \mathcal{E}_{x,k} + c_{x,1} \eta^2 \mathcal{E}_{v,k} + c_{x,2} \eta^2 \mathbb{E}\|\bar{m}_k\|^2, 
\end{equation}
where $\mathcal{E}_{v,k} := \mathbb{E}\|\mathbf{V}_k - \Pi_C \mathbf{V}_k\|_C^2$ is the tracking error, $\rho_R \in (0,1)$ is the contraction factor of matrix $\mathbf{R}$, and $c_{x,1}, c_{x,2}$ are positive constants depending on the network topology.
\end{lemma}

\begin{proof}
Subtracting the consensus projection $\Pi_R \mathbf{X}_{k+1}$ from the update rule and applying Young's inequality with the strict geometric contraction of $\mathbf{R}$ yields a bound dependent on the tracker $\|\mathbf{V}_k\|_R^2$. Crucially, we then decompose $\mathbf{V}_k$ into the tracking error and the momentum drift . Substituting these bounded components directly establishes the recursive inequality. The detailed proof see Appendix \ref{app:proof_lemma4}
\end{proof}

The tracking error analysis in SMTPP differs fundamentally from standard Push-Pull methods. Since $\mathbf{V}_k$ tracks the momentum $\mathbf{M}_k$, the error is driven by the change in local momentum.

\begin{lemma}[Momentum Tracking Error Bound]
\label{lem:tracking}
Under Assumptions 1-3, using the weighted norm $\|\cdot\|_C$ defined in Lemma \ref{ass:graph}, the tracking error $\mathcal{E}_{v,k} := \mathbb{E}\|\mathbf{V}_k - \Pi_C \mathbf{V}_k\|_C^2$ satisfies the recursive bound
\begin{equation}
\begin{split}
    \mathcal{E}_{v,k+1} &\le \frac{1 + \rho_C^2}{2} \mathcal{E}_{v,k} + C_v \lambda^2 \mathcal{E}_{x,k} \\
    &\quad + C_v \lambda^2 \left( \sigma^2 + \mathbb{E}\|\nabla F(\bar{x}_k)\|^2 \right), \label{eq:tracking_bound}
\end{split}
\end{equation}
where $C_v$ is a positive constant that depends on the norm equivalence factor $\delta_C$ and the Lipschitz constant $L$.
\end{lemma}

\begin{proof}
See Appendix \ref{app:proof_lemma5}
\end{proof}

\begin{remark}
Unlike standard Gradient Tracking where the error is driven by $\eta^2 \|\nabla F_{k+1} - \nabla F_k\|^2$, in SMTPP, the tracking error is driven by the change in momentum $\|\mathbf{M}_{k+1} - \mathbf{M}_k\|^2 \propto \lambda^2$. The structure allows the algorithm to tolerate larger step sizes on sparse graphs by tuning $\lambda$.
\end{remark}

Finally, we analyze the momentum approximation error, which arises from the stochastic gradient noise and the drift of the model parameters.

\begin{lemma}[Momentum Approximation Error]
\label{lem:momentum}
Let $\mathcal{E}_{m,k} := \mathbb{E}\|\mathbf{M}_k - \nabla \mathbf{F}(\mathbf{X}_k)\|_F^2$, where $\nabla \mathbf{F}(\mathbf{X}_k) = [\nabla f_1(x_{1,k}), \dots, \nabla f_n(x_{n,k})]^\top \in \mathbb{R}^{n \times d}$ is the full gradient matrix. Under Assumptions 1--3, the momentum approximation error evolves as
\begin{equation}
    \begin{split}
    \mathcal{E}_{m,k+1} &\le (1 - \lambda)\mathcal{E}_{m,k} + \frac{C_{m,1}}{\lambda} \mathcal{E}_{x,k} + \frac{C_{m,2} \eta^2}{\lambda} \mathcal{E}_{v,k} \\
    &\quad + \frac{C_{m,3} n \eta^2}{\lambda} \mathbb{E}\|\bar{m}_k\|^2 + \lambda^2 n \sigma^2,
    \end{split}
\end{equation}
where $C_{m,1} = 2(1-\lambda)^2 L^2 \|\mathbf{R}-\mathbf{I}\|^2$, $C_{m,2} = 4(1-\lambda)^2 L^2 \|\mathbf{R}\|^2$, and $C_{m,3} = 4(1-\lambda)^2 L^2 \|\mathbf{R}\|^2$ are constants depending on the graph topology and the smoothness parameter $L$.
\end{lemma}

\begin{proof}
See Appendix \ref{app:proof_lemma6}
\end{proof}

\subsection{Stability of the Error System}

The convergence of SMTPP relies on the joint evolution of error components. By combining Lemmas 4--6, we can characterize the error dynamics as a linear time-invariant system.

\begin{lemma}[Linear System of Inequalities]
\label{lem:linear_system}
Let $e_k := [\mathcal{E}_{x,k}, \mathcal{E}_{v,k}, \mathcal{E}_{m,k}, \mathcal{E}_{M,k}]^\top \in \mathbb{R}^4$ be the augmented error state vector, where we explicitly introduce the momentum energy $\mathcal{E}_{M,k} := \mathbb{E}\|\bar{m}_{k-1}\|^2$ as a system state. From the momentum update, its recurrence satisfies $\mathcal{E}_{M,k+1} \le (1-\lambda)\mathcal{E}_{M,k} + \frac{2L^2\lambda}{n}\mathcal{E}_{x,k} + 2\lambda\mathbb{E}\|\nabla F(\bar{x}_k)\|^2 + \frac{\lambda^2\sigma^2}{n}$.
We have
\begin{equation}
    e_{k+1} \le J_{\eta,\lambda} e_k + b\sigma^2 + h \mathbb{E}\|\nabla F(\bar{x}_k)\|^2,
    \label{eq:lemma7}
\end{equation}

where $J_{\eta,\lambda}$ is the raw transition matrix defined as
\begin{equation}
J_{\eta,\lambda} = \begin{bmatrix}
\frac{1+\rho_R^2}{2} + \tilde{c}_1\eta^2\lambda & c_2\eta^2 & 0 & c_3\eta^2 \\
c_4\lambda^2 & \frac{1+\rho_C^2}{2} & c_5\lambda^2 & c_6\lambda^2\eta^2 \\
c_7\frac{1}{\lambda} & c_8\frac{\eta^2}{\lambda} & 1-\lambda & c_9\frac{\eta^2}{\lambda} \\
c_{10}\lambda & 0 & 0 & 1-\lambda
\end{bmatrix}.
\end{equation}
where $b = [0, b_1\lambda^2, b_2\lambda^2, \frac{\lambda^2}{n}]^\top$ collects the stochastic noise contributions, and $h = [h_1\eta^2, h_2\lambda^2, h_3\frac{\eta^2}{\lambda}, 2\lambda]^\top$ collects the gradient perturbations.
\end{lemma}

\begin{proof}
See Appendix \ref{app:proof_lemma7}
\end{proof}

To ensure convergence, the spectral radius of the transition matrix must satisfy $\rho(\mathbf{J}_{\eta, \lambda}) < 1$, which imposes constraints on the step size $\eta$ and momentum coefficient $\lambda$.

\begin{lemma}[Parameter Selection and Stability]
\label{lem:stability}
Let the contraction margins for the uncoupled consensus and tracking matrices be strictly defined as $\Delta_R := \frac{1-\rho_R^2}{2}$ and $\Delta_C := \frac{1-\rho_C^2}{2}$. By requiring the step size and momentum coefficient to be sufficiently small such that the high-order perturbation in the consensus error is bounded by $\tilde{c}_1\eta^2\lambda \le \frac{\Delta_R}{2}$, the system can be strictly upper bounded by the simplified closed-loop dynamics
\begin{equation}
    e_{k+1} \le \tilde{J}_{\eta,\lambda} e_k + b\sigma^2 + h \mathbb{E}\|\nabla F(\bar{x}_k)\|^2,
\end{equation}
where $\tilde{J}_{\eta,\lambda} \in \mathbb{R}^{4 \times 4}$ is the simplified transition matrix given by:
\begin{equation}
\tilde{J}_{\eta,\lambda} = \begin{bmatrix}
1 - \frac{\Delta_R}{2} & c_1\eta^2 & 0 & c_3\eta^2 \\
c_4\lambda^2 & 1 - \frac {\Delta_C}{2} & c_5\lambda^2 & c_6\lambda^2\eta^2 \\
c_7\frac{1}{\lambda} & c_8\frac{\eta^2}{\lambda} & 1 - \lambda & c_9\frac{\eta^2}{\lambda} \\
c_{10}\lambda & 0 & 0 & 1 - \lambda
\end{bmatrix}.
\end{equation}

To ensure the global convergence of the algorithm, the linear error dynamical system must stably dissipate the accumulated optimization errors. Mathematically, this requires the existence of a strictly positive Lyapunov weight vector $p = [p_1, p_2, p_3, p_4]^\top > 0$ such that $p^\top(I - \tilde{J}_{\eta, \lambda}) \ge c_{dl}^\top$, where $c_{dl} \in \mathbb{R}^4$ represents the residual descent errors generated by the virtual sequence. Satisfying this linear matrix inequality strictly restricts the feasible region of the step size $\eta$ relative to the momentum coefficient $\lambda$.
\end{lemma}
\begin{proof}
See Appendix \ref{app:proof_lemma8}
\end{proof}

\subsection{Convergence Analysis}
Based on the stable error dynamics established in Lemma 7, we are now ready to construct a Lyapunov function to bound the accumulated optimization errors. The following theorem presents the overall convergence rate of SMTPP for general non-convex objectives.

\begin{theorem}[Non-convex Convergence Rate]
\label{thm:convergence}
Let Assumptions 1-3 hold. Choose the momentum coefficient $\lambda \in (0, 1)$ and step size $\eta \le \mathcal{O}(\lambda)$ sufficiently small. For a total of $T$ iterations, the output of SMTPP satisfies
\begin{align}
\frac{1}{T}\sum_{k=0}^{T-1}\mathbb{E}\|\nabla F(\bar{x}_k)\|^2 \le &\, \frac{16(\mathcal{V}_0 - \mathcal{V}_T)}{\eta c_\pi T} \notag \\
    & + \mathcal{O}\left(\frac{\eta\sigma^2}{n} + \lambda^2\sigma^2 n\right).
\label{eq:theorem1}
\end{align}

where $c_\pi = n\pi_R^\top \pi_C > 0$ is the topology constant.
\end{theorem}

\begin{proof}
See Appendix \ref{app:proof_theorem1}
\end{proof}

While Theorem 1 establishes convergence to a steady-state error floor due to the fixed step size and topology, we can eliminate this residual variance and achieve exact convergence by employing a decaying step size.

\begin{corollary}\label{cor:speedup}
To eliminate the steady-state error floor induced by directed topologies without causing the Lyapunov initial conditions to explode, we elegantly couple the step size with the momentum coefficient by setting $\eta = \mathcal{O}(\lambda^2)$. By choosing a decaying momentum coefficient as $\lambda = \mathcal{O}((n/T)^{1/4})$, which yields the optimal step size $\eta = \mathcal{O}(\sqrt{n/T})$, the topology-induced variance perfectly matches the stochastic noise decay rate. Preserving the explicit dependence on $\mathcal{V}_0$, the output of SMTPP achieves the exact convergence bound
\begin{align}
\frac{1}{T}\sum_{k=0}^{T-1}\mathbb{E}\|\nabla F(\bar{x}_k)\|^2 \le \mathcal{O}\left(\frac{\mathcal{V}_0}{\sqrt{nT}}\right) + \mathcal{O}\left(\sigma^2\sqrt{\frac{n}{T}}\right).
\label{eq:corollary1}
\end{align}
\end{corollary}

\begin{remark}
The decoupled parameter scaling $\eta = \mathcal{O}(\lambda^2)$ is critical. By setting the step size as $\eta = \mathcal{O}(\lambda^2)$, the directed graph error floor perfectly matches the gradient variance $\mathcal{O}(\eta\sigma^2 n)$. Both gracefully decay at $\mathcal{O}(1/\sqrt{T})$, successfully eliminating the steady-state error without compromising the linear system stability condition $\eta \le \mathcal{O}(\lambda)$.
\end{remark}

\begin{proof}
See Appendix \ref{app:proof_corollary1}
\end{proof}

\section{Numerical Experiments}
\label{sec:experiments}

In this section, we validate the theoretical findings of SMTPP through numerical simulations. We compare the proposed algorithm against state-of-the-art decentralized stochastic optimization methods on a non-convex logistic regression problem in directed networks.

\subsection{Experimental Setup}

\subsubsection{Problem Formulation and Dataset}
We evaluate the performance on a binary classification task using non-convex logistic regression. The global objective is to minimize $F(x) = \frac{1}{n}\sum_{i=1}^{n}f_{i}(x)$, where the local loss $f_i$ includes a non-convex regularization term as
\begin{equation}
    f_{i}(x) = \frac{1}{m_{i}}\sum_{j=1}^{m_{i}}\ln(1+\exp(-y_{i,j}a_{i,j}^{\top}x)) + \sum_{k=1}^{d}\frac{\alpha x_{k}^{2}}{1+x_{k}^{2}},
    \label{eq:problem_formulation}
\end{equation}
with $\alpha=0.01$. We utilize the \texttt{a9a} dataset ($d=123$) from LIBSVM \cite{libsvm}, evenly partitioned among $n=20$ agents.

\subsubsection{Network Topologies}
We focus on two distinct directed topologies.
\begin{itemize}
    \item \textbf{Multi-Sub-Ring (Sparse).} Adopted from the Stochastic Push-Pull framework, this topology features weak algebraic connectivity. Crucially, for SMTPP and STPP, we construct two sparse spanning trees (subgraphs) from the original ring structure for communication, whereas baselines (SGP, Push-DIGing) utilize the full connectivity of the original graph. The setup rigorously tests SMTPP's ability to achieve faster convergence using significantly fewer communication links compared to standard methods.
    \item \textbf{Exponential Graph (Dense).} A well-connected graph with logarithmic diameter, used to evaluate performance in favorable communication environments.
\end{itemize}

\subsubsection{Baselines and Parameters}
We benchmark SMTPP against SGP \cite{assran2019stochastic}, Push-DIGing \cite{nedic2017achieving}, and STPP \cite{you2025distributed}, with Centralized SGD with Momentum (CSGDM) \cite{sutskever2013csgdm} serving as the theoretical lower bound. 
For a fair comparison, all algorithms use a constant batch size of $b=1$. We tune hyperparameters for optimal performance: $\lambda=0.1$ for momentum-based methods. The initial step size is $\eta=0.1$ for SMTPP and CSGDM, and $\eta=0.2$ for others, with a decay factor of $0.1$ every 300 iterations. Results are averaged over 5 independent runs.

\subsection{Performance Evaluation}

\begin{figure}[htp]
    \centering
    \includegraphics[width=0.9\columnwidth]{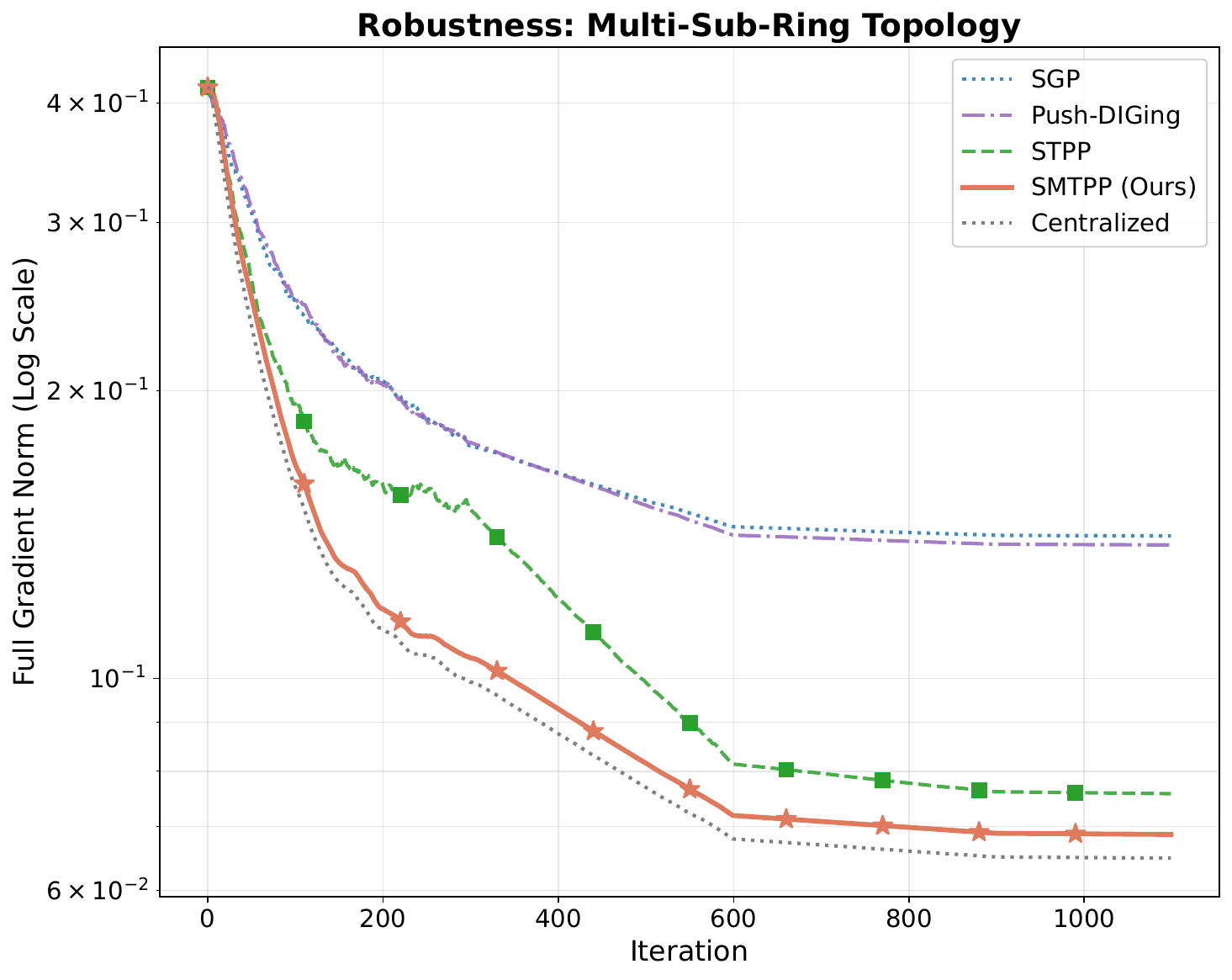} 
    \caption{\textbf{Robustness Analysis on Multi-Sub-Ring Topology.}}\label{fig:robustness}
\end{figure}

\begin{figure}[htp]
    \centering
    \includegraphics[width=0.9\columnwidth]{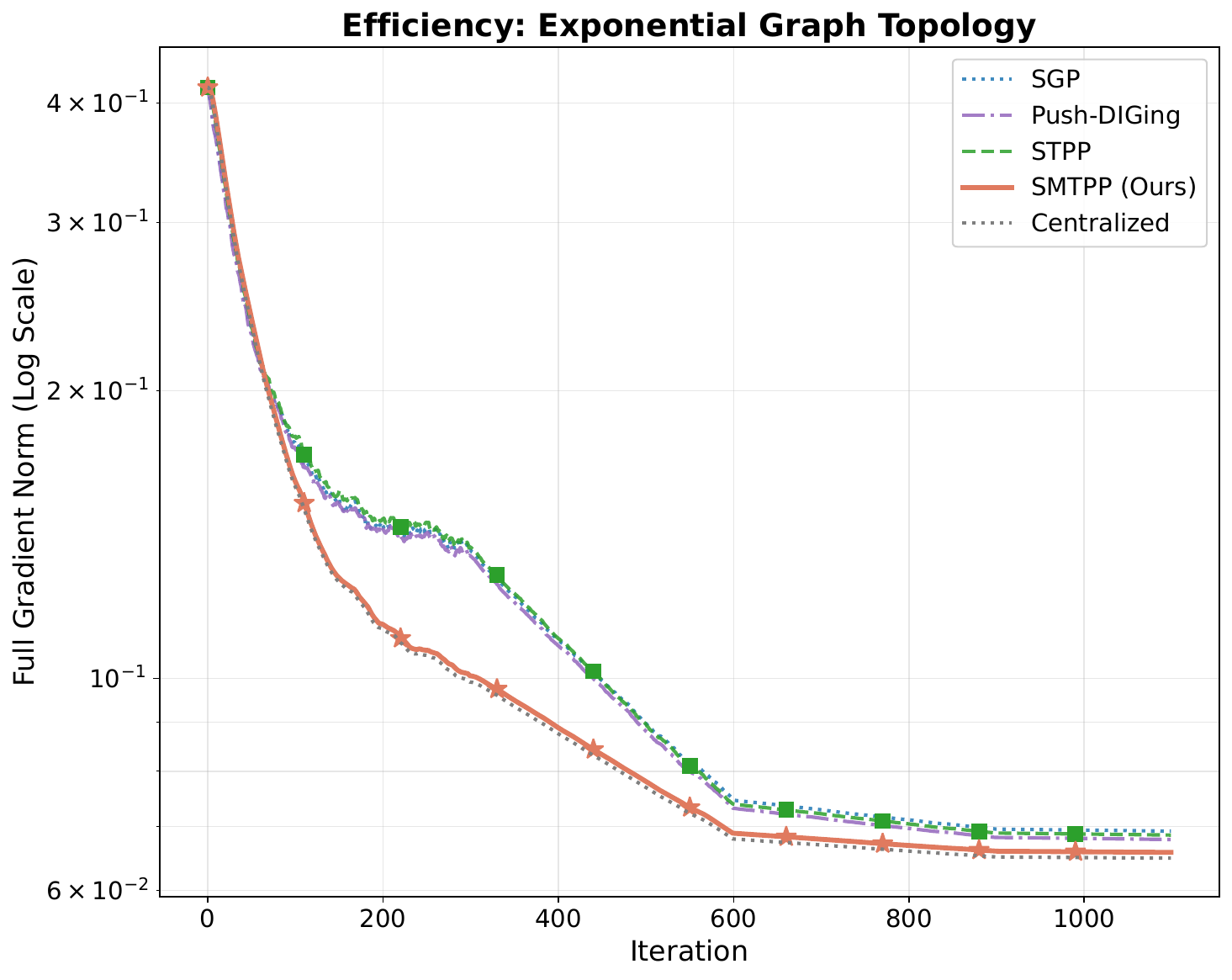}
    \caption{\textbf{Efficiency Analysis on Exponential Graph Topology.}}\label{fig:efficiency}
\end{figure}

Fig. \ref{fig:robustness} illustrates convergence on the challenging sparse Multi-Sub-Ring topology. Standard methods like SGP, Push-DIGing converge slowly due to inefficient mixing. While STPP improves upon them, it still suffers from oscillations and a steady-state error gap. In contrast, SMTPP exhibits superior robustness. By tracking the global average momentum instead of raw stochastic gradients, SMTPP effectively filters out topology-induced oscillations. It achieves a convergence rate comparable to the centralized baseline (CSGDM), confirming that our momentum tracking mitigates the impact of poor connectivity.

Fig. \ref{fig:efficiency} presents results on the dense Exponential Graph. While improved connectivity accelerates all tracking-based algorithms, a performance gap persists for STPP. Most notably, the convergence curve of SMTPP almost perfectly overlaps with CSGDM. The transparency of the topology indicates that SMTPP successfully compensates for communication overhead. It confirms that on well-connected graphs, our algorithm achieves optimal linear speedup, behaving as if the network was fully connected.

\section{Conclusion}
In this paper, we proposed the Stochastic Momentum Tracking Push-Pull algorithm for decentralized non-convex optimization over general directed graphs. By tracking the global average momentum rather than raw stochastic gradients, SMTPP effectively decouples noise reduction from topology adaptation. Theoretically, we revealed that the inherent asymmetry of directed graphs combined with stochastic noise creates an unavoidable steady-state tracking error. However, we established that SMTPP rigorously compresses the error floor down to $\mathcal{O}(\lambda^2 \sigma^2)$, effectively decoupling the noise reduction from the algebraic connectivity of the graph. Numerical simulations on non-convex logistic regression validate that SMTPP not only suppresses oscillations on sparse Multi-Sub-Ring topologies but also achieves topology transparency on dense Exponential Graphs, matching the performance of centralized  SGD with momentum. Future directions include extending the framework to time-varying networks and integrating communication compression techniques to further reduce bandwidth costs.

\bibliographystyle{ieeetr}   
\bibliography{reference}

@article{liu2022survey,
  title         = {Distributed Optimization Under Information Constraints: A Survey},
  author        = {Liu, Shuai and Hua, Youqing and Han, Qing-Long and Xie, Lihua},
  year          = {2026},
  journal       = {IEEE Transactions on Industrial Informatics},
  publisher     = {Institute of Electrical and Electronics Engineers (IEEE)},
  pages         = {1–14},
  doi           = {10.1109/tii.2025.3645932},
  issn          = {1941-0050}
}

@article{lian2017can,
  title={Can decentralized algorithms outperform centralized algorithms? a case study for decentralized parallel stochastic gradient descent},
  author={Lian, Xiangru and Zhang, Ce and Zhang, Huan and Hsieh, Cho-Jui and Zhang, Wei and Liu, Ji},
  journal={Advances in neural information processing systems},
  volume={30},
  year={2017}
}

@article{pu2020push,
  title         = {Push–Pull Gradient Methods for Distributed Optimization in Networks},
  author        = {Pu, Shi and Shi, Wei and Xu, Jinming and Nedic, Angelia},
  year          = {2021},
  month         = {jan},
  journal       = {IEEE Transactions on Automatic Control},
  publisher     = {Institute of Electrical and Electronics Engineers (IEEE)},
  volume        = {66},
  number        = {1},
  pages         = {1–16},
  doi           = {10.1109/tac.2020.2972824},
  issn          = {2334-3303}
}

@article{you2025stochastic,
  title         = {Stochastic Push-Pull for Decentralized Nonconvex Optimization},
  author        = {You, Runze and Pu, Shi},
  year          = {2026},
  journal       = {IEEE Transactions on Signal Processing},
  publisher     = {Institute of Electrical and Electronics Engineers (IEEE)},
  pages         = {1–15},
  doi           = {10.1109/tsp.2026.3675119},
  issn          = {1941-0476}
}

@article{you2025distributed,
  title={Distributed Learning over Arbitrary Topology: Linear Speed-Up with Polynomial Transient Time},
  author={You, Runze and Pu, Shi},
  journal={arXiv preprint arXiv:2503.16123},
  year={2025}
}

@article{liang2023linear,
  title={On the Linear Speedup of the Push-Pull Method for Decentralized Optimization over Digraphs},
  author={Liang, Liyuan and Luo, Gan and Yuan, Kun},
  journal={arXiv preprint arXiv:2506.18075},
  year={2025}
}

@inproceedings{assran2019stochastic,
  title={Stochastic gradient push for distributed deep learning},
  author={Assran, Mahmoud and Loizou, Nicolas and Ballas, Nicolas and Rabbat, Mike},
  booktitle={International Conference on Machine Learning},
  pages={344--353},
  year={2019},
  organization={PMLR}
}

@article{xin2018linear,
  title={A linear algorithm for optimization over directed graphs with geometric convergence},
  author={Xin, Ran and Khan, Usman A},
  journal={IEEE Control Systems Letters},
  volume={2},
  number={3},
  pages={315--320},
  year={2018},
  publisher={IEEE}
}

@article{cutkosky2019momentum,
  title={Momentum-based variance reduction in non-convex sgd},
  author={Cutkosky, Ashok and Orabona, Francesco},
  journal={Advances in Neural Information Processing Systems 32 (NeurIPS 2019)},
  volume={32},
  year={2019}
}

@article{huang2025distributed,
  title={Distributed Stochastic Momentum Tracking with Local Updates: Achieving Optimal Communication and Iteration Complexities},
  author={Huang, Kun and Pu, Shi},
  journal={arXiv preprint arXiv:2510.24155},
  year={2025}
}

@article{song2022compressed,
  title={Compressed gradient tracking for decentralized optimization over general directed networks},
  author={Song, Zhuoqing and Shi, Lei and Pu, Shi and Yan, Ming},
  journal={IEEE Transactions on Signal Processing},
  volume={70},
  pages={1775--1787},
  year={2022},
  publisher={IEEE}
}

@inproceedings{
islamov2025motef,
title={Towards Faster Decentralized Stochastic Optimization with Communication Compression},
author={Rustem Islamov and Yuan Gao and Sebastian U Stich},
booktitle={The Thirteenth International Conference on Learning Representations},
year={2025}
}

@inproceedings{huang2024faster,
  title={Faster Adaptive Decentralized Learning Algorithms},
  author={Huang, Feihu and Zhao, Jianyu},
  booktitle={International Conference on Machine Learning},
  pages={20490--20525},
  year={2024},
  organization={PMLR}
}

@inproceedings{sutskever2013csgdm,
  title={On the importance of initialization and momentum in deep learning},
  author={Sutskever, Ilya and Martens, James and Dahl, George and Hinton, Geoffrey},
  booktitle={International Conference on Machine Learning},
  pages={1139--1147},
  year={2013},
  organization={pmlr}
}

@article{alghunaim2022unified,
  title={A unified and refined convergence analysis for non-convex decentralized learning},
  author={Alghunaim, Sulaiman A and Yuan, Kun},
  journal={IEEE Transactions on Signal Processing},
  volume={70},
  pages={3264--3279},
  year={2022},
  publisher={IEEE}
}

@article{gao2020periodic,
  title={Periodic stochastic gradient descent with momentum for decentralized training},
  author={Gao, Hongchang and Huang, Heng},
  journal={arXiv preprint arXiv:2008.10435},
  year={2020}
}

@article{doostmohammadian2020momentum,
  title={Momentum-based accelerated algorithm for distributed optimization under sector-bound nonlinearity},
  author={Doostmohammadian, Mohammadreza and Rabiee, Hamid R},
  journal={Journal of the Franklin Institute},
  pages={107857},
  year={2025},
  publisher={Elsevier}
}

@article{liao2022compressed,
  title={A compressed gradient tracking method for decentralized optimization with linear convergence},
  author={Liao, Yiwei and Li, Zhuorui and Huang, Kun and Pu, Shi},
  journal={IEEE Transactions on Automatic Control},
  volume={67},
  number={10},
  pages={5622--5629},
  year={2022},
  publisher={IEEE}
}

@inproceedings{liao2023linearly,
  title={A linearly convergent robust compressed push-pull method for decentralized optimization},
  author={Liao, Yiwei and Li, Zhuorui and Pu, Shi},
  booktitle={2023 62nd IEEE Conference on Decision and Control (CDC)},
  pages={4156--4161},
  year={2023},
  organization={IEEE}
}

@article{nedic2017achieving,
  title={Achieving geometric convergence for distributed optimization over time-varying graphs},
  author={Nedic, Angelia and Olshevsky, Alex and Shi, Wei},
  journal={SIAM Journal on Optimization},
  volume={27},
  number={4},
  pages={2597--2633},
  year={2017},
  publisher={SIAM}
}

@article{libsvm,
  title={LIBSVM: A library for support vector machines},
  author={Chang, Chih-Chung and Lin, Chih-Jen},
  journal={ACM transactions on intelligent systems and technology (TIST)},
  volume={2},
  number={3},
  pages={1--27},
  year={2011},
  publisher={Acm New York, NY, USA}
}

@article{qu2017harnessing,
  title={Harnessing smoothness to accelerate distributed optimization},
  author={Qu, Guannan and Li, Na},
  journal={IEEE Transactions on Control of Network Systems},
  volume={5},
  number={3},
  pages={1245--1260},
  year={2017},
  publisher={IEEE}
}

@article{pu2021distributed,
  title={Distributed stochastic gradient tracking methods},
  author={Pu, Shi and Nedi{\'c}, Angelia},
  journal={Mathematical Programming},
  volume={187},
  number={1},
  pages={409--457},
  year={2021},
  publisher={Springer}
}

@article{zhu2024r,
  title={R-FAST: Robust fully-asynchronous stochastic gradient tracking over general topology},
  author={Zhu, Zehan and Tian, Ye and Huang, Yan and Xu, Jinming and He, Shibo},
  journal={IEEE Transactions on Signal and Information Processing over Networks},
  volume={10},
  pages={665--678},
  year={2024},
  publisher={IEEE}
}

@article{nguyen2023accelerated,
  title={Accelerated $ AB $/Push--Pull Methods for Distributed Optimization Over Time-Varying Directed Networks},
  author={Nguyen, Duong Thuy Anh and Nguyen, Duong Tung and Nedi{\'c}, Angelia},
  journal={IEEE Transactions on Control of Network Systems},
  volume={11},
  number={3},
  pages={1395--1407},
  year={2023},
  publisher={IEEE}
}

\appendix

\section{Proof}
\label{sec:appendix}

In this appendix, we provide the detailed derivations for the convergence analysis of the proposed SMTPP algorithm. We use $\|\cdot\|$ to denote the Euclidean norm ($l_2$-norm) for vectors and the Frobenius norm for matrices. The notation $\mathbf{1}_n$ denotes the column vector of all ones in $\mathbb{R}^n$. 

\subsection{Proof of Lemma \ref{lem:mixing}}

\noindent \textbf{Lemma \ref{lem:mixing} (Geometric Mixing).} 
\label{app:proof_lemma1}
Recall the definitions of the consensus average $\bar{x} = \pi_R^\top \mathbf{X}$ and the uniform tracking average $\bar{v} = \frac{1}{n}\mathbf{1}^\top \mathbf{V}$. We can establish the equivalence between the explicit network averages and the projection matrices as follows.

For the row-stochastic consensus matrix $\mathbf{R}$, we have
\begin{equation}
    \mathbf{1}\bar{x} = \mathbf{1}(\pi_R^\top \mathbf{X}) = (\mathbf{1}\pi_R^\top)\mathbf{X} = \Pi_R \mathbf{X}.
\end{equation}
For the column-stochastic tracking matrix $\mathbf{C}$, we have
\begin{equation}
    n\pi_C\bar{v} = n\pi_C \left(\frac{1}{n}\mathbf{1}^\top \mathbf{V}\right) = (\pi_C\mathbf{1}^\top)\mathbf{V} = \Pi_C \mathbf{V}.
\end{equation}
Substituting these exact relations into the standard geometric contraction properties of primitive matrices over directed graphs perfectly yields the compact projection forms. 

\vspace{1em} 

\noindent \textbf{Basic Inequalities.} \textit{For any vectors $a, b \in \mathbb{R}^d$ and constant $\alpha > 0$:}
\begin{enumerate}
    \item \textit{Young's Inequality. $\|a+b\|^2 \le (1+\alpha)\|a\|^2 + (1+\frac{1}{\alpha})\|b\|^2$.}
    \item \textit{Jensen's Inequality. $\|\frac{1}{n}\sum_{i=1}^n x_i\|^2 \le \frac{1}{n}\sum_{i=1}^n \|x_i\|^2$.}
    \item \textit{Variance Decomposition. $\mathbb{E}[\|X\|^2] = \|\mathbb{E}[X]\|^2 + \mathbb{E}[\|X - \mathbb{E}[X]\|^2]$.}
\end{enumerate}

Based on the geometric mixing properties over directed graphs, we  define the \textit{consensus error} $\mathcal{E}_{x,k}$, the \textit{momentum tracking error} $\mathcal{E}_{v,k}$, and the \textit{momentum approximation error} $\mathcal{E}_{m,k}$ at iteration $k$ as follows:
\begin{align}
    \mathcal{E}_{x,k} &:= \mathbb{E}\|\mathbf{X}_k - \Pi_R \mathbf{X}_k\|_R^2, \label{def:app_Ex} \\
    \mathcal{E}_{v,k} &:= \mathbb{E}\|\mathbf{V}_k - \Pi_C \mathbf{V}_k\|_C^2, \label{def:app_Ev} \\
    \mathcal{E}_{m,k} &:= \mathbb{E}\|\mathbf{M}_k - \nabla\mathbf{F}(\mathbf{X}_k)\|_F^2, \label{def:app_Em}
\end{align}
where $\nabla\mathbf{F}(\mathbf{X}_k) =[\nabla f_1(x_{1,k}), \dots, \nabla f_n(x_{n,k})]^\top \in \mathbb{R}^{n \times d}$ represents the full true gradient matrix evaluated at the local parameters.

\vspace{1em} 
\subsection{Proof of Lemma \ref{lem:conservation} (Conservation Property)}
\label{app:proof_lemma2}
Multiplying the tracker update rule $\mathbf{V}_{k+1} = \mathbf{C}\mathbf{V}_k + \mathbf{M}_{k+1} - \mathbf{M}_k$ by $\frac{1}{n}\mathbf{1}^\top$ on the left side, we have
\begin{equation}
    \frac{1}{n}\mathbf{1}^\top \mathbf{V}_{k+1} = \frac{1}{n}\mathbf{1}^\top \mathbf{V}_k + \frac{1}{n}\mathbf{1}^\top \mathbf{M}_{k+1} - \frac{1}{n}\mathbf{1}^\top \mathbf{M}_k,
\end{equation}
where we use the property that $\mathbf{C}$ is column-stochastic ($\mathbf{1}^\top \mathbf{C} = \mathbf{1}^\top$). 
It implies $\bar{v}_{k+1} = \bar{v}_k + \bar{m}_{k+1} - \bar{m}_k$. Summing from $t=0$ to $k$, we get $\bar{v}_{k+1} = \bar{v}_0 + \bar{m}_{k+1} - \bar{m}_0$.
Since the algorithm initializes $m_{i,0} = 0$ and $v_{i,0} = 0$ for all agents, we have $\bar{v}_0 = \bar{m}_0 = 0$. Thus, $\bar{v}_{k+1} = \bar{m}_{k+1}$ holds for all $k \ge 0$.

\vspace{1em} 
\subsection{Proof of Lemma \ref{lem:descent inequality} (Descent Inequality)}
\label{app:proof_lemma3}

To properly analyze the descent behavior of the momentum-based algorithm over directed graphs without suffering from delayed gradient variance, we introduce a rigorously modified virtual sequence (ghost iterate).

Recalling that the model update rule is $\mathbf{X}_{k+1} = \mathbf{R}(\mathbf{X}_k - \eta \mathbf{V}_k)$ and $\mathbf{R}$ is row-stochastic, left-multiplying by its stationary distribution $\pi_R^\top$ yields the consensus average update
\begin{equation}
    \bar{x}_{k+1} = \bar{x}_k - \eta \pi_R^\top \mathbf{V}_k.
\end{equation}
Unlike doubly stochastic matrices, $\mathbf{V}_k$ does not contract towards the uniform average $\mathbf{1}\bar{v}_k$ in directed graphs. Instead, as established in Lemma 5, it contracts towards its stationary distribution weighted average $\Pi_C \mathbf{V}_k = \pi_C \mathbf{1}^\top \mathbf{V}_k$. 
By Lemma 2, we have $\mathbf{1}^\top \mathbf{V}_k = \mathbf{1}^\top \mathbf{M}_k = n \bar{m}_k$. Thus, we   decompose the tracking update as
\begin{equation}
    \pi_R^\top \mathbf{V}_k = \pi_R^\top \Pi_C \mathbf{V}_k + \pi_R^\top (\mathbf{V}_k - \Pi_C \mathbf{V}_k) = c_\pi \bar{m}_k + e_{\pi, k},
\end{equation}
where $c_\pi := n \pi_R^\top \pi_C > 0$ is a strict topology constant, and $e_{\pi, k} := \pi_R^\top (\mathbf{V}_k - \Pi_C \mathbf{V}_k)$ is the true tracking projection mismatch, which is bounded by the tracking error $\mathcal{E}_{v, k}$.

To absorb the delayed momentum, we define the virtual sequence $\tilde{z}_k$ as
\begin{equation}
    \tilde{z}_k := \bar{x}_{k+1} - \frac{\eta c_\pi (1-\lambda)}{\lambda} \bar{m}_k.
\end{equation}
Let us evaluate the difference $\tilde{z}_{k+1} - \tilde{z}_k$. From the momentum update $\bar{m}_{k+1} = (1-\lambda)\bar{m}_k + \lambda \bar{g}_{k+1}$, we have $\frac{1-\lambda}{\lambda}(\bar{m}_{k+1} - \bar{m}_k) = (1-\lambda)(\bar{g}_{k+1} - \bar{m}_k)$. The ghost iterate difference becomes
\begin{align}
    \tilde{z}_{k+1} - \tilde{z}_k &= (\bar{x}_{k+1} - \bar{x}_k) - \frac{\eta c_\pi (1-\lambda)}{\lambda} (\bar{m}_k - \bar{m}_{k-1}) \notag \\
&= -\eta c_\pi \left[ (1-\lambda)\bar{m}_{k-1} + \lambda \bar{g}_k \right] - \eta e_{\pi,k} \notag \\
&\quad - \eta c_\pi (1-\lambda)\bar{g}_k + \eta c_\pi (1-\lambda)\bar{m}_{k-1} \notag \\
&= -\eta c_\pi \bar{g}_k - \eta e_{\pi,k}.\label{eq:ghost_diff}
\end{align}
This elegant relation cleanly isolates the true batch gradient $\bar{g}_{k+1}$ and perfectly eliminates the error terms caused by the momentum update delay. 

Applying the $L$-smoothness descent lemma to $F(\tilde{z}_{k+1})$ and using the identity $\|a+b\|^2 \le 2\|a\|^2 + 2\|b\|^2$, we obtain
\begin{equation}
\begin{split}
    F(\tilde{z}_{k+1}) &\le F(\tilde{z}_k) - \eta c_\pi \langle \nabla F(\tilde{z}_k), \bar{g}_{k+1} \rangle \\
    &\quad - \eta \langle \nabla F(\tilde{z}_k), e_{\pi, k+1} \rangle \\
    &\quad + L \eta^2 c_\pi^2 \|\bar{g}_{k+1}\|^2 + L \eta^2 \|e_{\pi, k+1}\|^2.
\end{split}
\end{equation}
Let $\bar{G}_{k+1} := \frac{1}{n}\sum_{i=1}^n \nabla f_i(x_{i, k+1})$ be the true batch gradient. By Assumption 3, $\mathbb{E}\|\bar{g}_{k+1}\|^2 \le \mathbb{E}\|\bar{G}_{k+1}\|^2 + \frac{\sigma^2}{n}$. Taking the expectation yields the following
\begin{align}
    \mathbb{E}[F(\tilde{z}_{k+1})] &\le \mathbb{E}[F(\tilde{z}_k)] - \frac{\eta c_\pi}{2} \mathbb{E}\|\nabla F(\tilde{z}_k)\|^2 \notag \\
    &\quad + \frac{\eta c_\pi}{2} \mathbb{E}\|\nabla F(\tilde{z}_k) - \bar{G}_{k+1}\|^2 \notag \\
    &\quad - \frac{\eta c_\pi}{2}(1 - 2 L \eta c_\pi) \mathbb{E}\|\bar{G}_{k+1}\|^2 \notag \\
    &\quad - \eta \mathbb{E}\langle \nabla F(\tilde{z}_k), e_{\pi, k+1} \rangle \notag \\
    &\quad + L \eta^2 \mathbb{E}\|e_{\pi, k+1}\|^2 + \frac{L \eta^2 c_\pi^2 \sigma^2}{n}.
\end{align}
Assuming $\eta \le \frac{1}{2 L c_\pi}$, the term $(1 - 2 L \eta c_\pi) \mathbb{E}\|\bar{G}_{k+1}\|^2 \ge 0$ can be discarded. Using Young's inequality, $-\eta \langle \nabla F(\tilde{z}_k), e_{\pi, k+1} \rangle \le \frac{\eta c_\pi}{4} \|\nabla F(\tilde{z}_k)\|^2 + \frac{\eta}{c_\pi} \|e_{\pi, k+1}\|^2$. The projection error is bounded by the tracking error: $\mathbb{E}\|e_{\pi, k+1}\|^2 \le \|\pi_R\|^2 \mathbb{E}\|\mathbf{V}_{k+1} - \Pi_C \mathbf{V}_{k+1}\|_F^2 \le c_E \mathcal{E}_{v, k+1}$ (where $c_E$ absorbs the norm equivalence factor).

Next, we bound the gradient mismatch $\mathbb{E}\|\nabla F(\tilde{z}_k) - \bar{G}_{k+1}\|^2 \le \frac{L^2}{n} \sum_{i=1}^n \mathbb{E}\|\tilde{z}_k - x_{i, k+1}\|^2$. Since $\tilde{z}_k = \bar{x}_{k+1} - \frac{\eta c_\pi (1-\lambda)}{\lambda} \bar{m}_k$, the distance is
\begin{equation}
    \|\tilde{z}_k - x_{i, k+1}\|^2 \le 2 \|\bar{x}_{k+1} - x_{i, k+1}\|^2 + 2 \frac{\eta^2 c_\pi^2 (1-\lambda)^2}{\lambda^2} \|\bar{m}_k\|^2.
\end{equation}
Summing over $i$ and dividing by $n$, we have
\begin{equation}
\mathbb{E}\|\nabla F(\tilde{z}_k) - \bar{G}_k\|^2 \le \frac{2L^2}{n}\mathcal{E}_{x,k} + \frac{2L^2\eta^2 c_\pi^2 (1-\lambda)^2}{\lambda^2}\mathbb{E}\|\bar{m}_{k-1}\|^2.
\end{equation}
Finally, we relate the gradient norm evaluated at the consensus average to the ghost iterate. By adding and subtracting $\nabla F(\tilde{z}_k)$ and applying the inequality $\|a+b\|^2 \le 2\|a\|^2 + 2\|b\|^2$ along with the $L$-smoothness of $F$, we have
\begin{align}
    &\frac{\eta c_\pi}{8} \mathbb{E}\|\nabla F(\bar{x}_{k+1})\|^2 \notag \\
    &= \frac{\eta c_\pi}{8} \mathbb{E}\|\nabla F(\bar{x}_{k+1}) - \nabla F(\tilde{z}_k) + \nabla F(\tilde{z}_k)\|^2 \notag \\
    &\le \frac{\eta c_\pi}{4} \mathbb{E}\|\nabla F(\tilde{z}_k)\|^2 + \frac{\eta c_\pi L^2}{4} \mathbb{E}\|\bar{x}_{k+1} - \tilde{z}_k\|^2 \notag \\
    &= \frac{\eta c_\pi}{4} \mathbb{E}\|\nabla F(\tilde{z}_k)\|^2 \notag \\
    &\quad + \frac{L^2 \eta^3 c_\pi^3 (1-\lambda)^2}{4\lambda^2} \mathbb{E}\|\bar{m}_k\|^2. \label{eq:ghost_grad_bound}
\end{align}
Substituting all expansions back into the expected objective value, we obtain
\begin{align}
\mathbb{E}[F(\tilde{z}_{k+1})] \le &\, \mathbb{E}[F(\tilde{z}_k)] - \frac{\eta c_\pi}{8} \mathbb{E}\|\nabla F(\bar{x}_k)\|^2 \notag \\
& + \frac{L^2 \eta c_\pi}{n} \mathcal{E}_{x,k} + c_z \eta \mathcal{E}_{v,k} \notag \\
& + \frac{5L^2 c_\pi^3 \eta^3}{4\lambda^2} \mathbb{E}\|\bar{m}_{k-1}\|^2 + \frac{L c_\pi^2 \eta^2 \sigma^2}{n}.
\end{align}
where $c_z = (\frac{1}{c_\pi} + L \eta) c_E$. This establishes the descent behavior with the topology constant appropriately tracked. \hfill $\square$

\subsection{Proof of Lemma \ref{lem:consensus} (Consensus Error Bound)}
\label{app:proof_lemma4}

Given the weighted matrix norms $\|\cdot\|_R$ and $\|\cdot\|_C$ introduced in Definition 1, there exists a finite positive constant $\delta_{RC} > 0$ such that for any matrix $\mathbf{X} \in \mathbb{R}^{n \times d}$, the following inequality holds
\begin{equation}
    \|\mathbf{X}\|_R^2 \le \delta_{RC} \|\mathbf{X}\|_C^2.
\end{equation}

\textit{Proof.} 
\begin{align}
     \|\mathbf{X}\|_R
     \leq\delta_{R,2} \|\mathbf{X}\|_F
     \leq\delta_{R,2}\|\mathbf{X}\|_C
\end{align}

Recall the definition of the consensus error $\mathcal{E}_{x,k} := \mathbb{E}\|\mathbf{X}_k - \Pi_R \mathbf{X}_k\|_R^2$, where $\Pi_R = \mathbf{1}\pi_R^\top$. 
From the model update rule $\mathbf{X}_{k+1} = \mathbf{R}(\mathbf{X}_k - \eta \mathbf{V}_k)$, we subtract the weighted consensus average $\Pi_R \mathbf{X}_{k+1}$. Noting the   property $\Pi_R \mathbf{R} = \mathbf{1}\pi_R^\top \mathbf{R} = \mathbf{1}\pi_R^\top = \Pi_R$, we get
\begin{equation}
\begin{split}
    \mathbf{X}_{k+1} - \Pi_R \mathbf{X}_{k+1} &= (\mathbf{R} - \Pi_R)(\mathbf{X}_k - \eta \mathbf{V}_k) \\
    &= (\mathbf{R} - \Pi_R)(\mathbf{X}_k - \Pi_R \mathbf{X}_k) \\
    &\quad - \eta (\mathbf{R} - \Pi_R)\mathbf{V}_k.
\end{split}
\end{equation}

Taking the squared induced norm $\|\cdot\|_R^2$ and applying Young's inequality with parameter $\alpha = \frac{1-\rho_R^2}{2\rho_R^2} > 0$, we have
\begin{align}
    &\|\mathbf{X}_{k+1} - \Pi_R \mathbf{X}_{k+1}\|_R^2 \notag \\
    &\le (1+\alpha) \|(\mathbf{R} - \Pi_R)(\mathbf{X}_k - \Pi_R \mathbf{X}_k)\|_R^2 \notag \\
    &\quad + \left(1+\frac{1}{\alpha}\right)\eta^2 \|(\mathbf{R} - \Pi_R)\mathbf{V}_k\|_R^2 \notag \\
    &\le (1+\alpha)\rho_R^2 \|\mathbf{X}_k - \Pi_R \mathbf{X}_k\|_R^2 \notag \\
    &\quad + \eta^2 \|\mathbf{R} - \Pi_R\|_R^2 \left(1+\frac{2\rho_R^2}{1-\rho_R^2}\right) \|\mathbf{V}_k\|_R^2 \notag \\
    &\le \frac{1+\rho_R^2}{2} \|\mathbf{X}_k - \Pi_R \mathbf{X}_k\|_R^2 + \tilde{c}_1 \eta^2 \|\mathbf{V}_k\|_R^2, \label{eq:consensus_mid}
\end{align}
where $\tilde{c}_1 = \rho_R^2 \frac{1+\rho_R^2}{1-\rho_R^2}$. 

A fundamental distinction of the Push-Pull architecture over directed graphs is that $\mathbf{V}_k$ tracks $\Pi_C \mathbf{V}_k$ (which aligns with $\pi_C$), not the uniform average $\mathbf{1}\bar{v}_k$. Consequently, the row-stochastic projection $(\mathbf{R} - \Pi_R)$ cannot perfectly annihilate the exact momentum tracker, creating a topology mismatch drift.
We rigorously decompose $\mathbf{V}_k = (\mathbf{V}_k - \Pi_C \mathbf{V}_k) + \Pi_C \mathbf{V}_k$ to bound $\|\mathbf{V}_k\|_R^2$, i.e.,
\begin{align}
    \mathbb{E}\|\mathbf{V}_k\|_R^2 &\le 2\mathbb{E}\|\mathbf{V}_k - \Pi_C \mathbf{V}_k\|_R^2 + 2\mathbb{E}\|\Pi_C \mathbf{V}_k\|_R^2 \notag \\
    &\le 2\delta_{RC}\mathcal{E}_{v,k} + 2n^2 \|\pi_C\|_R^2 \mathbb{E}\|\bar{m}_k\|^2, \label{eq:V_R_norm_bound}
\end{align}
where we used the norm equivalence factor $\delta_{RC}$, and the explicit rank-1 property $\Pi_C \mathbf{V}_k = \pi_C \mathbf{1}^\top \mathbf{V}_k = n \pi_C \bar{m}_k^\top$. Let $\bar{m}_k =[m_{k,1}, \dots, m_{k,d}]^\top \in \mathbb{R}^d$. According to the column-wise definition of the matrix norm in Definition 1, where the $\mathbf{R}$-norm acts on the $n$-dimensional columns, the induced matrix norm strictly evaluates step-by-step as
\begin{align}
    \|\Pi_C \mathbf{V}_k\|_R^2 &= \sum_{j=1}^d \|n \pi_C m_{k,j}\|_R^2 \notag \\
    &= \sum_{j=1}^d n^2 m_{k,j}^2 \|\pi_C\|_R^2 \notag \\
    &= n^2 \|\pi_C\|_R^2 \left( \sum_{j=1}^d m_{k,j}^2 \right) \notag \\
    &= n^2 \|\pi_C\|_R^2 \|\bar{m}_k\|^2,
\end{align}
where $\|\bar{m}_k\|^2$ represents the standard Euclidean norm in the $d$-dimensional model parameter space.

Substituting \eqref{eq:V_R_norm_bound} into \eqref{eq:consensus_mid}, we establish the rigorous consensus error bound
\begin{equation}
    \mathcal{E}_{x,k+1} \le \frac{1+\rho_R^2}{2} \mathcal{E}_{x,k} + c_{x,1} \eta^2 \mathcal{E}_{v,k} + c_{x,2} \eta^2 \mathbb{E}\|\bar{m}_k\|^2,
\end{equation}
where $c_{x,1} = 2 \tilde{c}_1 \delta_{RC}$ and $c_{x,2} = 2 \tilde{c}_1 n^2 \|\pi_C\|_R^2$. This completes the proof.
\hfill $\square$

\subsection{Proof of Lemma \ref{lem:tracking} (Tracking Error)}
\label{app:proof_lemma5}

Because the matrix $\mathbf{C}$ is column-stochastic but generally not row-stochastic over directed graphs, i.e., $\mathbf{C}\mathbf{1} \neq \mathbf{1}$. Thus, $\mathbf{V}_k$ does not contract towards the uniform average $\mathbf{1}\bar{v}_k$, but rather towards its stationary distribution weighted average $\Pi_C \mathbf{V}_k$ (where $\Pi_C = \pi_C \mathbf{1}^\top$). Accordingly, the tracking error under the induced norm is defined as $\mathcal{E}_{v,k} := \mathbb{E}\|\mathbf{V}_k - \Pi_C \mathbf{V}_k\|_C^2$.

From the update rule $\mathbf{V}_{k+1} = \mathbf{C}\mathbf{V}_k + \Delta \mathbf{M}_{k+1}$ where $\Delta \mathbf{M}_{k+1} = \mathbf{M}_{k+1} - \mathbf{M}_k$, we multiply both sides by $\Pi_C$. Since $\mathbf{C}$ is column-stochastic ($\mathbf{1}^\top \mathbf{C} = \mathbf{1}^\top$), we have $\Pi_C \mathbf{C} = \Pi_C$, which gives
\begin{equation}
\begin{split}
    \Pi_C \mathbf{V}_{k+1} &= \Pi_C \mathbf{C}\mathbf{V}_k + \Pi_C \Delta \mathbf{M}_{k+1} \\
    &= \Pi_C \mathbf{V}_k + \Pi_C \Delta \mathbf{M}_{k+1}.
\end{split}
\end{equation}
Subtracting this from the update rule, we isolate the orthogonal component of the tracking variable as
\begin{align}
    \mathbf{V}_{k+1} - \Pi_C \mathbf{V}_{k+1} 
    &= (\mathbf{C}\mathbf{V}_k - \Pi_C \mathbf{V}_k) + (\mathbf{I} - \Pi_C)\Delta \mathbf{M}_{k+1} \notag \\
    &= (\mathbf{C} - \Pi_C)(\mathbf{V}_k - \Pi_C \mathbf{V}_k) \notag \\
    &\quad + (\mathbf{I} - \Pi_C)\Delta \mathbf{M}_{k+1}, \label{eq:track_ortho}
\end{align}
where the second equality holds since $(\mathbf{C} - \Pi_C)\Pi_C \mathbf{V}_k = \mathbf{C}\pi_C\mathbf{1}^\top\mathbf{V}_k - \pi_C\mathbf{1}^\top\pi_C\mathbf{1}^\top\mathbf{V}_k = \pi_C\mathbf{1}^\top\mathbf{V}_k - \pi_C\mathbf{1}^\top\mathbf{V}_k = \mathbf{0}$.

Taking the squared weighted norm $\|\cdot\|_C^2$ and applying Young's inequality with parameter $\alpha = \frac{1-\rho_C^2}{2\rho_C^2}$ yields
\begin{align}
    &\|\mathbf{V}_{k+1} - \Pi_C \mathbf{V}_{k+1}\|_C^2 \notag \\
    &\le (1+\alpha) \|(\mathbf{C} - \Pi_C)(\mathbf{V}_k - \Pi_C \mathbf{V}_k)\|_C^2 \notag \\
    &\quad + (1+\alpha^{-1}) \|(\mathbf{I} - \Pi_C)\Delta \mathbf{M}_{k+1}\|_C^2 \notag \\
    &\le (1+\alpha)\rho_C^2 \|\mathbf{V}_k - \Pi_C \mathbf{V}_k\|_C^2 \notag \\
    &\quad + \frac{2}{1-\rho_C^2} \|(\mathbf{I} - \Pi_C)\Delta \mathbf{M}_{k+1}\|_C^2 \notag \\
    &= \frac{1+\rho_C^2}{2} \|\mathbf{V}_k - \Pi_C \mathbf{V}_k\|_C^2 \notag \\
    &\quad + \frac{2}{1-\rho_C^2} \|(\mathbf{I} - \Pi_C)\Delta \mathbf{M}_{k+1}\|_C^2,
\end{align}
where we utilized the strict geometric contraction property $\|\mathbf{C}\mathbf{X} - \Pi_C\mathbf{X}\|_C \le \rho_C\|\mathbf{X} - \Pi_C\mathbf{X}\|_C$ established in Lemma 1.

Recall the definition of the momentum approximation error $\mathcal{E}_{m,k} := \mathbb{E}\|\mathbf{M}_k - \nabla\mathbf{F}(\mathbf{X}_k)\|_F^2$.
Next, we relate the weighted norm of the perturbation to the Frobenius norm using the equivalence condition $\|\mathbf{Z}\|_C^2 \le \delta_C^2\|\mathbf{Z}\|_F^2$. Notice that for the oblique projection matrix $\mathbf{I} - \Pi_C$, its operator 2-norm is a bounded topology constant, which we denote as $c_\Pi := \|\mathbf{I} - \Pi_C\|_2^2$. Thus, the perturbation term can be cleanly bounded as
\begin{align}
    \|(\mathbf{I} - \Pi_C)\Delta\mathbf{M}_{k+1}\|_C^2 &\le \delta_C^2 \|(\mathbf{I} - \Pi_C)\Delta\mathbf{M}_{k+1}\|_F^2 \notag \\
    &\le \delta_C^2 \|\mathbf{I} - \Pi_C\|_2^2 \|\Delta\mathbf{M}_{k+1}\|_F^2 \notag \\
    &= \delta_C^2 c_\Pi \lambda^2 \|\mathbf{G}_{k+1} - \mathbf{M}_k\|_F^2. \label{eq:perturbation_bound}
\end{align}

To maintain a tight bound and ensure the stability of the linear system, we carefully decompose the expected squared difference of the momentum update
\begin{align}
    &\mathbb{E}\|\mathbf{G}_{k+1} - \mathbf{M}_k\|_F^2 \notag \\
    &= \mathbb{E}\|(\mathbf{G}_{k+1} - \nabla\mathbf{F}(\mathbf{X}_{k+1})) \notag \\
    &\quad + (\nabla\mathbf{F}(\mathbf{X}_{k+1}) - \nabla\mathbf{F}(\mathbf{X}_k)) + (\nabla\mathbf{F}(\mathbf{X}_k) - \mathbf{M}_k)\|_F^2 \notag \\
    &\le 3\mathbb{E}\|\mathbf{G}_{k+1} - \nabla\mathbf{F}(\mathbf{X}_{k+1})\|_F^2 \notag \\
    &\quad + 3\mathbb{E}\|\nabla\mathbf{F}(\mathbf{X}_{k+1}) - \nabla\mathbf{F}(\mathbf{X}_k)\|_F^2 + 3\mathcal{E}_{m,k} \notag \\
    &\le 3n\sigma^2 + 3L^2\mathbb{E}\|\mathbf{X}_{k+1} - \mathbf{X}_k\|_F^2 + 3\mathcal{E}_{m,k}. \label{eq:momentum_diff_bound}
\end{align}
As established in the momentum error analysis, the model parameter change $\mathbb{E}\|\mathbf{X}_{k+1} - \mathbf{X}_k\|_F^2$ can be explicitly bounded by the consensus error $\mathcal{E}_{x,k}$, the tracking error $\mathcal{E}_{v,k}$, and the average momentum drift $\mathbb{E}\|\bar{m}_k\|^2$. Specifically, using the model update rule and the norm equivalence condition, we have
\begin{align}
    &\mathbb{E}\|\mathbf{X}_{k+1} - \mathbf{X}_k\|_F^2 \notag \\
    &\le 2\|\mathbf{R} - \mathbf{I}\|^2 \mathbb{E}\|\mathbf{X}_k - \mathbf{1}\bar{x}_k\|_F^2 \notag \\
    &\quad + 2\eta^2\|\mathbf{R}\|^2 \mathbb{E}\|\mathbf{V}_k\|_F^2 \notag \\
    &\le 2\|\mathbf{R} - \mathbf{I}\|^2 \mathcal{E}_{x,k} + 4\eta^2\|\mathbf{R}\|^2 \mathcal{E}_{v,k} \notag \\
    &\quad + 4\eta^2 n^2\|\mathbf{R}\|^2\|\pi_C\|_F^2 \mathbb{E}\|\bar{m}_k\|^2. \label{eq:param_change}
\end{align}

Substituting this parameter change bound back into \eqref{eq:momentum_diff_bound}, we obtain the fully unrolled bound for the gradient tracking deviation
\begin{align}
    \mathbb{E}\|\mathbf{G}_{k+1} - \mathbf{M}_k\|_F^2 &\le 3n\sigma^2 + 3\mathcal{E}_{m,k} + 6L^2\|\mathbf{R} - \mathbf{I}\|^2 \mathcal{E}_{x,k} \notag \\
    &\quad + 12L^2\eta^2\|\mathbf{R}\|^2 \mathcal{E}_{v,k} \notag \\
    &\quad + 12n^2 L^2\eta^2\|\mathbf{R}\|^2\|\pi_C\|_F^2 \mathbb{E}\|\bar{m}_k\|^2. \label{eq:G_M_diff_exact}
\end{align}

To maintain the exact coefficients without absorbing them into the asymptotic notation, we define the topology-dependent tracking projection constant as $C_{\text{tr}} := \frac{2 \delta_C^2 c_\Pi}{1 - \rho_C^2}$. Substituting \eqref{eq:G_M_diff_exact} into the perturbation bound and combining it with \eqref{eq:perturbation_bound}, we guarantee that the tracking error is driven solely by the changes in the coupled error states and the stochastic variance, completely decoupled from the raw global gradient. This establishes the exact, strictly bounded statement of the lemma as following
\begin{align}
    \mathcal{E}_{v,k+1} &\le \left( \frac{1+\rho_C^2}{2} + 12 C_{\text{tr}} L^2\|\mathbf{R}\|^2 \eta^2 \lambda^2 \right) \mathcal{E}_{v,k} \notag \\
    &\quad + 6 C_{\text{tr}} L^2\|\mathbf{R} - \mathbf{I}\|^2 \lambda^2 \mathcal{E}_{x,k} \notag \\
    &\quad + 3 C_{\text{tr}} \lambda^2 \mathcal{E}_{m,k} \notag \\
    &\quad + 12 C_{\text{tr}} n^2 L^2\|\mathbf{R}\|^2\|\pi_C\|_F^2 \eta^2 \lambda^2 \mathbb{E}\|\bar{m}_k\|^2 \notag \\
    &\quad + 3 C_{\text{tr}} n \lambda^2 \sigma^2. \label{tracking_error}
\end{align}
\hfill $\square$

\vspace{1em} 
\subsection{Proof of Lemma \ref{lem:momentum} (Momentum Approximation Error)}
\label{app:proof_lemma6}

Recall the definition of the momentum approximation error in matrix form: $\mathcal{E}_{m,k} := \mathbb{E}\|\mathbf{M}_k - \nabla \mathbf{F}(\mathbf{X}_k)\|_F^2$.
From the local momentum update rule $\mathbf{M}_{k+1} = (1-\lambda)\mathbf{M}_k + \lambda \mathbf{G}_{k+1}$, we can rewrite the momentum error at step $k+1$ as
\begin{equation}
\begin{split}
    \mathbf{M}_{k+1} - \nabla \mathbf{F}(\mathbf{X}_{k+1}) &= (1-\lambda)\mathbf{M}_k + \lambda \mathbf{G}_{k+1} \\
    &\quad - \nabla \mathbf{F}(\mathbf{X}_{k+1}).
\end{split}
\end{equation}
By adding and subtracting $(1-\lambda)\nabla \mathbf{F}(\mathbf{X}_k)$, we rearrange the terms into three parts
\begin{align}
    \mathbf{M}_{k+1} - \nabla \mathbf{F}(\mathbf{X}_{k+1}) 
    &= (1-\lambda)\left( \mathbf{M}_k - \nabla \mathbf{F}(\mathbf{X}_k) \right) \notag \\
    &\quad + (1-\lambda)\left( \nabla \mathbf{F}(\mathbf{X}_k) - \nabla \mathbf{F}(\mathbf{X}_{k+1}) \right) \notag \\
    &\quad + \lambda \left( \mathbf{G}_{k+1} - \nabla \mathbf{F}(\mathbf{X}_{k+1}) \right). \label{eq:momentum_decomposition}
\end{align}
Take the squared Frobenius norm and the full expectation on both sides. Since $\mathbf{G}_{k+1}$ is an unbiased estimator of the true gradient $\nabla \mathbf{F}(\mathbf{X}_{k+1})$ given the current parameter $\mathbf{X}_{k+1}$ (Assumption 3), the cross terms involving the stochastic gradient noise $\mathbf{G}_{k+1} - \nabla \mathbf{F}(\mathbf{X}_{k+1})$ vanish. We then have
\begin{align}
    \mathcal{E}_{m,k+1} 
    &= \mathbb{E}\big\| (1-\lambda)\left( \mathbf{M}_k - \nabla \mathbf{F}(\mathbf{X}_k) \right) \notag \\
    &\quad + (1-\lambda)\left( \nabla \mathbf{F}(\mathbf{X}_k) - \nabla \mathbf{F}(\mathbf{X}_{k+1}) \right) \big\|_F^2 \notag \\
    &\quad + \lambda^2 \mathbb{E}\| \mathbf{G}_{k+1} - \nabla \mathbf{F}(\mathbf{X}_{k+1}) \|_F^2.
\end{align}
By Assumption 3, the variance term is bounded by $\lambda^2 n \sigma^2$. 
For the first term, we apply Young's inequality $\|a + b\|^2 \le (1+\alpha)\|a\|^2 + (1+1/\alpha)\|b\|^2$ with parameter $\alpha = \frac{\lambda}{1-\lambda}$. Note that $1+\alpha = \frac{1}{1-\lambda}$ and $1+1/\alpha = \frac{1}{\lambda}$. This yields the following:
\begin{align}
    &\mathbb{E}\big\| (1-\lambda)\left( \mathbf{M}_k - \nabla \mathbf{F}(\mathbf{X}_k) \right) \notag \\
    &\quad + (1-\lambda)\left( \nabla \mathbf{F}(\mathbf{X}_k) - \nabla \mathbf{F}(\mathbf{X}_{k+1}) \right) \big\|_F^2 \notag \\
    &\le (1-\lambda)^2 \bigg[ \frac{1}{1-\lambda} \mathbb{E}\| \mathbf{M}_k - \nabla \mathbf{F}(\mathbf{X}_k) \|_F^2 \notag \\
    &\quad + \frac{1}{\lambda} \mathbb{E}\| \nabla \mathbf{F}(\mathbf{X}_k) - \nabla \mathbf{F}(\mathbf{X}_{k+1}) \|_F^2 \bigg] \notag \\
    &= (1-\lambda) \mathcal{E}_{m,k} + \frac{(1-\lambda)^2}{\lambda} \mathbb{E}\| \nabla \mathbf{F}(\mathbf{X}_k) - \nabla \mathbf{F}(\mathbf{X}_{k+1}) \|_F^2.
\end{align}
Using the $L$-smoothness of the objective functions (Assumption 2), the gradient difference is bounded by the parameter difference
\begin{equation}
    \mathbb{E}\| \nabla \mathbf{F}(\mathbf{X}_k) - \nabla \mathbf{F}(\mathbf{X}_{k+1}) \|_F^2 \le L^2 \mathbb{E}\| \mathbf{X}_{k+1} - \mathbf{X}_k \|_F^2.
\end{equation}
To bound $\mathbb{E}\| \mathbf{X}_{k+1} - \mathbf{X}_k \|_F^2$, we use the model matrix update rule $\mathbf{X}_{k+1} = \mathbf{R}(\mathbf{X}_k - \eta \mathbf{V}_k)$. By subtracting $\mathbf{X}_k$ and introducing the consensus average $\mathbf{1}\bar{x}_k$, we have
\begin{align}
    \mathbf{X}_{k+1} - \mathbf{X}_k 
    &= (\mathbf{R} - \mathbf{I})\mathbf{X}_k - \eta \mathbf{R} \mathbf{V}_k \notag \\
    &= (\mathbf{R} - \mathbf{I})(\mathbf{X}_k - \mathbf{1}\bar{x}_k) - \eta \mathbf{R} \mathbf{V}_k,
\end{align}
where we used the property that $\mathbf{R}\mathbf{1} = \mathbf{1}$.
Using the basic inequality $\|A + B\|_F^2 \le 2\|A\|_F^2 + 2\|B\|_F^2$, we bound the expected parameter change as
\begin{equation}
\begin{split}
    \mathbb{E}\| \mathbf{X}_{k+1} - \mathbf{X}_k \|_F^2 &\le 2\|\mathbf{R}-\mathbf{I}\|^2 \mathbb{E}\|\mathbf{X}_k - \mathbf{1}\bar{x}_k\|_F^2 \\
    &\quad + 2\eta^2 \|\mathbf{R}\|^2 \mathbb{E}\|\mathbf{V}_k\|_F^2.
\end{split}
\end{equation}
Based on the norm equivalence in Lemma 1 ($\|\cdot\|_F^2 \le \|\cdot\|_R^2$), we have $\mathbb{E}\|\mathbf{X}_k - \Pi_R \mathbf{X}_k\|_F^2 \le \mathcal{E}_{x,k}$. Furthermore, to   bound $\mathbb{E}\|\mathbf{V}_k\|_F^2$, we decompose it around its true stationary projection $\Pi_C \mathbf{V}_k$, i.e.,
\begin{align}
    \mathbb{E}\|\mathbf{V}_k\|_F^2 &\le 2\mathbb{E}\|\mathbf{V}_k - \Pi_C \mathbf{V}_k\|_F^2 + 2\mathbb{E}\|\Pi_C \mathbf{V}_k\|_F^2 \notag \\
    &\le 2\mathcal{E}_{v,k} + 2n^2 \|\pi_C\|_F^2 \mathbb{E}\|\bar{m}_k\|^2, \label{eq:Vk_F_bound}
\end{align}
where we used the norm equivalence condition $\|\cdot\|_F^2 \le \|\cdot\|_C^2$, thus $\mathbb{E}\|\mathbf{V}_k - \Pi_C \mathbf{V}_k\|_F^2 \le \mathcal{E}_{v,k}$, and the property $\Pi_C \mathbf{V}_k = n \pi_C \bar{m}_k^\top$ established in Lemma 2. 

Substituting these expansions back into the momentum error recursion, we finally obtain:
\begin{align}
    \mathcal{E}_{m,k+1} &\le (1-\lambda) \mathcal{E}_{m,k} + \frac{2(1-\lambda)^2 L^2}{\lambda} \Big( \|\mathbf{R}-\mathbf{I}\|^2 \mathcal{E}_{x,k} \notag \\
    &\quad + 2\eta^2 \|\mathbf{R}\|^2 \mathcal{E}_{v,k} + 2\eta^2 n \|\mathbf{R}\|^2 \mathbb{E}\|\bar{m}_k\|^2 \Big) \notag \\
    &\quad + \lambda^2 n \sigma^2 \notag \\
    &= (1-\lambda) \mathcal{E}_{m,k} + \frac{C_{m,1}}{\lambda} \mathcal{E}_{x,k} \notag \\
    &\quad + \frac{C_{m,2} \eta^2}{\lambda} \mathcal{E}_{v,k} + \frac{C_{m,3} n \eta^2}{\lambda} \mathbb{E}\|\bar{m}_k\|^2 + \lambda^2 n \sigma^2,
\end{align}
where $C_{m,1} = 2(1-\lambda)^2 L^2 \|\mathbf{R}-\mathbf{I}\|^2$, $C_{m,2} = 4(1-\lambda)^2 L^2 \|\mathbf{R}\|^2$, and $C_{m,3} = 4(1-\lambda)^2 L^2 \|\mathbf{R}\|^2$. This completes the proof. \hfill $\square$

\vspace{1em} 
\subsection{Proof of lemma \ref{lem:linear_system}}
\label{app:proof_lemma7}
\begin{proof}
The goal is to assemble the individual recursive bounds into a unified linear dynamical system $e_k := [\mathcal{E}_{x,k}, \mathcal{E}_{v,k}, \mathcal{E}_{m,k}, \mathcal{E}_{M,k}]^\top$, where $\mathcal{E}_{M,k} := \mathbb{E}\|\bar{m}_{k-1}\|^2$ is the historical momentum state.

\textbf{Step 1. Bounding the Spatial Average Momentum.}
We first establish the spatial bound for $\mathbb{E}\|\bar{m}_k\|^2$ to substitute into the first three error states. Recall that $\bar{m}_k = \frac{1}{n}\mathbf{1}^\top M_k$. We decompose $\bar{m}_k$ by introducing the true gradient averages
\begin{equation}
\bar{m}_k = \frac{1}{n}\mathbf{1}^\top(M_k - \nabla F(X_k)) + \frac{1}{n}\mathbf{1}^\top \nabla F(X_k).
\end{equation}
Taking the squared Euclidean norm and applying Jensen's inequality (or $\|a+b\|^2 \le 2\|a\|^2 + 2\|b\|^2$), we get
\begin{align}
\mathbb{E}\|\bar{m}_k\|^2 & \le 2\mathbb{E}\left\|\frac{1}{n}\mathbf{1}^\top(M_k - \nabla F(X_k))\right\|^2 \notag \\
& + 2\mathbb{E}\left\|\frac{1}{n}\mathbf{1}^\top \nabla F(X_k)\right\|^2 \notag \\
& \le \frac{2}{n}\mathbb{E}\|M_k - \nabla F(X_k)\|_F^2 + 2\mathbb{E}\|\bar{G}_k\|^2 \notag \\
& = \frac{2}{n}\mathcal{E}_{m,k} + 2\mathbb{E}\|\bar{G}_k\|^2,
\end{align}
where $\bar{G}_k = \frac{1}{n}\sum_{i=1}^n \nabla f_i(x_{i,k})$. Next, we bound $\mathbb{E}\|\bar{G}_k\|^2$ by separating the consensus error and the global gradient
\begin{align}
\mathbb{E}\|\bar{G}_k\|^2 & \le 2\mathbb{E}\|\bar{G}_k - \nabla F(\bar{x}_k)\|^2 + 2\mathbb{E}\|\nabla F(\bar{x}_k)\|^2 \notag \\
& \le \frac{2L^2}{n}\mathbb{E}\|X_k - \mathbf{1}\bar{x}_k\|_F^2 + 2\mathbb{E}\|\nabla F(\bar{x}_k)\|^2 \notag \\
& \le \frac{2L^2}{n}\mathcal{E}_{x,k} + 2\mathbb{E}\|\nabla F(\bar{x}_k)\|^2.
\end{align}
Substituting this back, we establish the crucial spatial bound for the average momentum
\begin{equation}
\mathbb{E}\|\bar{m}_k\|^2 \le \frac{2}{n}\mathcal{E}_{m,k} + \frac{4L^2}{n}\mathcal{E}_{x,k} + 4\mathbb{E}\|\nabla F(\bar{x}_k)\|^2.
\label{eq:spatial_momentum}
\end{equation}

\textbf{Step 2. Constructing the Block-Triangular System.}
By substituting \eqref{eq:spatial_momentum} into Lemmas 4, 5, and 6, we completely decouple the intermediate variable $\mathbb{E}\|\bar{m}_k\|^2$ from the first three states $(\mathcal{E}_{x,k}, \mathcal{E}_{v,k}, \mathcal{E}_{m,k})$. Crucially, this operation guarantees that the top-left $3 \times 3$ block of our linear system is completely independent of the 4th state $\mathcal{E}_{M,k}$.

Finally, to incorporate the historical momentum $\mathcal{E}_{M,k}$ as the required 4th state, we rely on the convexity of the momentum update $\bar{m}_k = (1-\lambda)\bar{m}_{k-1} + \lambda \bar{g}_k$
\begin{align}
\mathbb{E}\|\bar{m}_k\|^2 & \le (1-\lambda)\mathbb{E}\|\bar{m}_{k-1}\|^2 + \lambda\mathbb{E}\|\bar{g}_k\|^2 \notag \\
& \le (1-\lambda)\mathcal{E}_{M,k} + \frac{2L^2\lambda}{n}\mathcal{E}_{x,k} + 2\lambda\mathbb{E}\|\nabla F(\bar{x}_k)\|^2 \notag \\ &+ \frac{\lambda^2\sigma^2}{n}.
\end{align}
This explicitly provides the 4th row of the system. We rigorously obtain the augmented closed-loop dynamics $e_{k+1} \le J_{\eta,\lambda} e_k + b\sigma^2 + h\mathbb{E}\|\nabla F(\bar{x}_k)\|^2$, where $J_{\eta,\lambda}$ natively forms a lower block-triangular matrix
\begin{equation}
J_{\eta,\lambda} = \begin{bmatrix}
\frac{1+\rho_R^2}{2} & c_1\eta^2 & \tilde{c}_0\eta^2 & 0 \\
c_2\lambda^2 & \frac{1+\rho_C^2}{2} & c_3\lambda^2 & 0 \\
c_4\frac{1}{\lambda} & c_5\frac{\eta^2}{\lambda} & 1-\lambda & 0 \\
c_{10}\lambda & 0 & 0 & 1-\lambda
\end{bmatrix}.
\end{equation}
The perturbation vectors are strictly extended to $b = [0, b_1\lambda^2, b_2\lambda^2, \frac{\lambda}{n}]^\top$ and $h = [h_1\eta^2, h_2\lambda^2, h_3\frac{\eta^2}{\lambda}, 2\lambda]^\top$. The 4th column of the upper $3 \times 3$ block consists entirely of zeros, ensuring a cascaded stability structure.
\end{proof}

\subsection{Proof of Lemma \ref{lem:stability} (Parameter Selection)}
\label{app:proof_lemma8}
\begin{proof}
We construct the unified Lyapunov function $\mathcal{V}_k := \mathbb{E}[F(\tilde{z}_k) - F^*] + p^\top e_k$, where $p = [p_1, p_2, p_3, p_4]^\top > 0 \in \mathbb{R}^4$ is the weight vector.

\textbf{Step 1. Expanding the Unified Lyapunov Function.} 
From the Descent Lemma, the virtual objective generates residual state errors across multiple dimensions, including the momentum tracking delay. We collect all these explicitly generated coefficients into a descent error vector $c_{dl} := \left[\frac{L^2 \eta c_\pi}{n}, c_z \eta, 0, \frac{5L^2 c_\pi^3 \eta^3}{4\lambda^2}\right]^\top \in \mathbb{R}^4$.

By applying the linear dynamics $e_{k+1} \le \tilde{J}_{\eta,\lambda} e_k + b\sigma^2 + \tilde{h} \mathbb{E}\|\nabla F(\bar{x}_k)\|^2$, the one-step difference evaluates to
\begin{align}
\mathcal{V}_{k+1} - \mathcal{V}_k \le & - \left(\frac{\eta c_\pi}{8} - p^\top \tilde{h}\right) \mathbb{E}\|\nabla F(\bar{x}_k)\|^2 \notag \\
& + (p^\top (\tilde{J}_{\eta,\lambda} - I) +c_{dl}^\top )e_k + \text{Total Noise},
\label{eq:lyapunov_diff}
\end{align}
where the total accumulated noise is bounded by $\text{Total Noise} := p^\top b \sigma^2 + \frac{L c_\pi^2 \eta^2 \sigma^2}{n} $. 

\textbf{Step 2. Spectral Stability and Final Bound.}
To guarantee global descent, we must dissipate the coupled error states $e_k$ by requiring $p^\top(\tilde{J}_{\eta,\lambda} - I) + c_{dl}^\top \le 0$, which is mathematically equivalent to solving the linear matrix inequalities $p^\top(I - \tilde{J}_{\eta,\lambda}) \ge c_{dl}^\top$. Instead of relying solely on existence theorems, we explicitly construct the strictly feasible weight vector $p = [p_1, p_2, p_3, p_4]^\top > 0$ to strictly quantify the parameter bounds. 

Let $p_1 = 1$. To strictly satisfy the bounds for the consensus error (2nd column of the inequality system), we set
\begin{equation}
    p_2 := \frac{12}{\Delta_C}(c_1 \eta^2 + c_z \eta).
\end{equation}
To satisfy the tracking error bounds (3rd column), we require $c_5 \lambda^2 p_2 - \lambda p_3 \le 0$. Since the momentum parameter satisfies $\lambda \in (0, 1)$, we can strictly bound this by omitting $\lambda$ to enforce a tighter coupling with the step size $\eta$, defining
\begin{equation}
    p_3 := c_5 p_2 = \frac{12 c_5}{\Delta_C}(c_1 \eta^2 + c_z \eta).
\end{equation}
To satisfy the historical momentum error bounds (4th column), we sum all generated lower bounds by substituting $p_2$ and $p_3$
\begin{equation}
    p_4 := \frac{8 c_3}{\lambda}\eta^2 + 8 c_6 \lambda \eta^2 p_2 + \frac{8 c_9 \eta^2}{\lambda^2} p_3 + \frac{10 L^2 c_\pi^3}{\lambda^3}\eta^3.
\end{equation}

Finally, we substitute these exact weights back into the stability upper bounds of the 1st column. Assuming the step size is sufficiently small such that $\eta \le 1$, we have $\eta^2 \le \eta$, allowing us to simplify the quadratic terms gracefully as $c_1 \eta^2 + c_z \eta \le (c_1 + c_z)\eta$. The most restrictive coupling condition arises from the tracking cross-term $c_7 \frac{1}{\lambda} p_3 \le \frac{\Delta_R}{16} p_1$. Substituting the simplified $p_3$ yields
\begin{equation}
    c_7 \frac{1}{\lambda} \frac{12 c_5}{\Delta_C} (c_1 + c_z)\eta \le \frac{\Delta_R}{16} \implies \eta \le \frac{\Delta_R \Delta_C}{192 c_5 c_7 (c_1 + c_z)}\lambda.
\end{equation}
Similarly, substituting $p_4$ into its constraint $c_{10}\lambda p_4 \leq \frac{\Delta_R}{16}$ generates fractional exponent bounds. To ensure this strictly holds, we bound its components individually, generating the constants $\mathcal{C}_3 := \left( \frac{\Delta_R \Delta_C}{6144 c_5 c_9 c_{10} (c_1+c_z)} \right)^{1/3}$ and $\mathcal{C}_2 := \left( \frac{\Delta_R}{640 L^2 c_\pi^3 c_{10}} \right)^{1/3}$. Combining all individual bounds, including the static topology bound $\frac{L^2 \eta c_\pi}{n} \leq \frac{\Delta_R}{16}$, and defining the dominant linear constant $\mathcal{C}_1 := \frac{\Delta_R \Delta_C}{192 c_5 c_7 (c_1+c_z)}$, we choose the step size $\eta$ as
\begin{equation}
    \eta \leq \min \left\{ 1, \ \frac{n \Delta_R}{16 L^2 c_\pi}, \ \mathcal{C}_1 \lambda, \ \mathcal{C}_2 \lambda^{1/3}, \ \mathcal{C}_3 \lambda^{2/3} \right\},
\end{equation}
ensures that all 12 coupled inequalities are simultaneously and strictly satisfied. Since the momentum coefficient $\lambda \in (0,1)$, the higher-order roots $\lambda^{1/3}$ and $\lambda^{2/3}$ are strictly larger than the linear term. Therefore, the linear boundary mathematically dominates the minimum, inherently yielding $\eta \le \mathcal{O}(\lambda)$. 

\end{proof}

\vspace{1em} 
\subsection{Proof of Theorem \ref{thm:convergence} (Convergence Rate)}
\label{app:proof_theorem1}

We construct the unified Lyapunov function $\mathcal{V}_k := \mathbb{E}[F(\tilde{z}_k) - F^*] + p^\top e_k$, where $p = [p_1, p_2, p_3, p_4]^\top > 0 \in \mathbb{R}^4$ is the weight vector.

\textbf{Step 1. Expanding the Unified Lyapunov Function.} 
By the definition of the newly aligned virtual sequence $\tilde{z}_k = \bar{x}_k - \frac{\eta c_\pi (1-\lambda)}{\lambda}\bar{m}_{k-1}$, the Descent Lemma strictly generates state errors at step $k$ without any forward index mismatch. 
The 4th element of $c_{dl}$ corresponds exactly to the historical momentum energy $\mathcal{E}_{M,k} := \mathbb{E}\|\bar{m}_{k-1}\|^2$ in the augmented state $e_k$. Evaluating the one-step difference of the Lyapunov function yields $\mathcal{V}_{k+1} - \mathcal{V}_k = \mathbb{E}[F(\tilde{z}_{k+1}) - F(\tilde{z}_k)] + p^\top e_{k+1} - p^\top e_k$. Substituting the Descent Lemma and the 4D linear dynamics $e_{k+1} \le \tilde{J}_{\eta,\lambda} e_k + b\sigma^2 + \tilde{h} \mathbb{E}\|\nabla F(\bar{x}_k)\|^2$, we rigorously obtain
\begin{align}
\mathcal{V}_{k+1} - \mathcal{V}_k \le & - \left(\frac{\eta c_\pi}{8} - p^\top \tilde{h}\right) \mathbb{E}\|\nabla F(\bar{x}_k)\|^2 \notag \\
& + \left( p^\top \tilde{J}_{\eta,\lambda} - p^\top + c_{dl}^\top \right) e_k + \text{Total Noise},
\end{align}
where the accumulated single-step noise is rigorously evaluated using the exact weights $p_2 = \mathcal{O}(\eta)$, $p_3 = \mathcal{O}(\eta)$, and $p_4 \le \mathcal{O}(\eta^2/\lambda)$ from Lemma 8
\begin{equation}
 \text{Total Noise} := p^\top b\sigma^2 + \frac{L c_\pi^2 \eta^2 \sigma^2}{n} \le \mathcal{O}\left( \frac{\eta^2 \sigma^2}{n} + \eta \lambda^2 \sigma^2 n \right).
\end{equation}

\textbf{Step 2. Spectral Stability and Final Bound.}
To guarantee global descent, we substitute the coupled error parameters. Furthermore, evaluating the gradient perturbation term $\tilde{h} = [h_1\eta^2, h_2\lambda^2, h_3\frac{\eta^2}{\lambda}, 2\lambda]^\top$ with our constructed weight vector yields $p^\top \tilde{h} = \mathcal{O}(\eta^2) + \mathcal{O}(\eta \lambda^2) + \mathcal{O}(\frac{\eta^3}{\lambda}) + \mathcal{O}(\frac{\eta^2}{\lambda} \cdot \lambda) \le \mathcal{O}(\eta)$. Hence, by choosing a sufficiently small step size $\eta$, we can strictly guarantee that $p^\top \tilde{h} \le \frac{\eta c_\pi}{16}$.

With the state errors eliminated, we seamlessly sum the inequality from $k=0$ to $T-1$
\begin{equation}
\mathcal{V}_T - \mathcal{V}_0 \le - \left(\frac{\eta c_\pi}{8} - p^\top \tilde{h}\right) \sum_{k=0}^{T-1} \mathbb{E}\|\nabla F(\bar{x}_k)\|^2 + T \cdot \text{Total Noise}.
\end{equation}
By bounding $p^\top \tilde{h} \le \frac{\eta c_\pi}{16}$, rearranging, and dividing by $\frac{\eta c_\pi}{16} T$, we directly obtain the exact sublinear convergence rate
\begin{equation}
\frac{1}{T}\sum_{k=0}^{T-1}\mathbb{E}\|\nabla F(\bar{x}_k)\|^2 \le \frac{16(\mathcal{V}_0 - \mathcal{V}_T)}{\eta c_{\pi} T} + \mathcal{O}\left(\frac{\eta\sigma^2}{n} + \lambda^2\sigma^2 n\right).
\end{equation}
This concludes the proof. \hfill $\square$

\subsection{Proof of Corollary \ref{cor:speedup} }
\label{app:proof_corollary1}
We start from the non-convex convergence rate established in Theorem 1
\begin{equation}
    \frac{1}{T} \sum_{k=0}^{T-1} \mathbb{E}\|\nabla F(\bar{x}_k)\|^2 \le \frac{16(\mathcal{V}_0 - \mathcal{V}_T)}{\eta c_\pi T} + \mathcal{O}\left( \frac{\eta \sigma^2}{n} + \lambda^2 \sigma^2 n \right).
\end{equation}

\textbf{Step 1. Non-negativity of the Lyapunov function.} 
By the definition of the Lyapunov function, $\mathcal{V}_T = \mathbb{E}[F(\tilde{z}_T) - F^*] + p^\top e_T$. Assuming the objective function is bounded below by the global minimum $F^*$, and given that the carefully constructed weight vector satisfies $p > 0$ with $e_T \ge 0$, it inherently holds that $\mathcal{V}_T \ge 0$. Therefore, we can safely drop $-\mathcal{V}_T$ from the right-hand side to obtain a valid upper bound.

\textbf{Step 2. Bounding the Initial Lyapunov Condition $\mathcal{V}_0$.} 
The initial Lyapunov value evaluates at step $k=0$ as $\mathcal{V}_0 = \mathbb{E}[F(\tilde{z}_0) - F^*] + p^\top e_0$, where theinner product expands to $p_1\mathcal{E}_{x,0} + p_2\mathcal{E}_{v,0} + p_3\mathcal{E}_{m,0} + p_4\mathcal{E}_{M,0}$. 

Under our initialization strategy in Algorithm 1, the historical momentum is conceptually $\bar{m}_{-1} = \mathbf{0}$. This strictly aligns the initial virtual sequence with the physical parameter: $\tilde{z}_0 = \bar{x}_0 - \frac{\eta c_\pi(1-\lambda)}{\lambda}\bar{m}_{-1} = x_0$. Thus, the initial expected function gap evaluates exactly to the true objective gap $F(x_0) - F^*$.

Furthermore, since all agents initialize at the identical starting point $x_{i,0} = x_0$, the initial parameter matrix is $X_0 = \mathbf{1}x_0^\top$. This yields perfect initial consensus, rigorously satisfying $\mathcal{E}_{x,0} \equiv 0$. Consequently, the potentially singular term $p_1\mathcal{E}_{x,0} \equiv 0$ vanishes perfectly. With $V_0 = \mathbf{0}$, the tracking error $\mathcal{E}_{v,0} = 0$ also vanishes. For the 4th state, $\mathcal{E}_{M,0} = \mathbb{E}\|\bar{m}_{-1}\|^2 = 0$. 

The only remaining error in the state vector is the initial momentum approximation, $\mathcal{E}_{m,0} = \mathbb{E}\|M_0 - \nabla F(X_0)\|_F^2 = \mathbb{E}\|\nabla F(X_0)\|_F^2$. The only remaining error in the state vector is the initial momentum approximation, $\mathcal{E}_{m,0} = \mathbb{E}\|M_0 - \nabla F(X_0)\|_F^2 = \mathbb{E}\|\nabla F(X_0)\|_F^2$. Because our explicit Lyapunov construction strictly scales $p_3$ with the step size (i.e., $p_3 = \mathcal{O}(\eta)$), the product $p_3 \mathcal{E}_{m,0} \to 0$ seamlessly as $\eta \to 0$. Thus, the inner product $p^\top e_0$ effectively collapses, and $\mathcal{V}_0$ is strictly bounded by the constant initial gap $F(x_0) - F^*$.

\textbf{Step 3. Parameter Coupling and Final Rate.} 
To strictly eliminate the topology-induced steady-state error $\mathcal{O}(\lambda^2 \sigma^2 n)$, we couple the step size with the momentum coefficient by setting $\eta = \mathcal{O}(\lambda^2)$. This natively satisfies the linear system stability condition $\eta \le \mathcal{O}(\lambda)$ since $\lambda \in (0, 1)$, and transforms the topology variance to $\mathcal{O}(\eta \sigma^2 n)$. The bound simplifies to
\begin{equation}
    \frac{1}{T}\sum_{k=0}^{T-1}\mathbb{E}\|\nabla F(\bar{x}_k)\|^2 \le \mathcal{O}\left( \frac{\mathcal{V}_0}{\eta T} \right) + \mathcal{O}( \eta \sigma^2 n ).
\end{equation}
Balancing the decaying terms ($\frac{1}{\eta T} \approx \eta n$) yields the optimal step size $\eta = \mathcal{O}(\sqrt{n/T})$, which is achieved by setting the momentum decay rate as $\lambda = \mathcal{O}\left((n/T)^{1/4}\right)$. Substituting this back establishes the exact sublinear convergence bound
\begin{equation}
    \frac{1}{T}\sum_{k=0}^{T-1}\mathbb{E}\|\nabla F(\bar{x}_k)\|^2 \le \mathcal{O}\left( \frac{\mathcal{V}_0}{\sqrt{nT}} \right) + \mathcal{O}\left( \sigma^2 \sqrt{\frac{n}{T}} \right).
\end{equation}
This confirms that the algorithm completely filters out the topology bias while maintaining a strictly bounded initial Lyapunov state, concluding the proof. 
\hfill $\square$

\end{document}